\documentclass[11pt,a4paper, reqno]{amsart}
\usepackage{fullpage}
\usepackage{amssymb}
\usepackage{mathtools}
\usepackage{amsthm}
\usepackage{caption}
\usepackage{subcaption}
\usepackage{amsmath}
\usepackage{comment}
\usepackage{wrapfig}
\usepackage[usenames]{color}
\usepackage[all]{xy}
\usepackage{graphicx,epsfig,float}
\usepackage{tikz-cd}
\usetikzlibrary{positioning}
\usetikzlibrary{shapes,arrows.meta,calc}
\usepackage{tikz}

\usepackage{chngcntr}
\counterwithin{figure}{section}
\usepackage{thmtools}
\usepackage{thm-restate}

\usepackage[colorlinks=true,
linkcolor=webgreen,
filecolor=webbrown,
citecolor=webgreen]{hyperref}
\definecolor{webgreen}{rgb}{0,.5,0}
\definecolor{webbrown}{rgb}{.6,0,0}
\definecolor{red}{rgb}{1,0,0}
\usepackage{cleveref}
\usepackage{color}
\usepackage{diagbox}

\tikzset{
	block/.style={rectangle, draw,  text width=2em,
		text centered, rounded corners, minimum height=1.5em},
	arrow/.style={-{Stealth[]}}
}
\tikzset{
	every node/.style={font=\sffamily\small},
	main node/.style={thick,circle,draw,font=\sffamily\Large}
}

\newcolumntype{P}[1]{>{\centering\arraybackslash}p{#1}}
\newcolumntype{M}[1]{>{\centering\arraybackslash}m{#1}}

\newcommand{\Z}{\mathbb{Z}}

\newcommand{\K}{\mathcal{K}}
\newcommand{\W}{\mathcal{W}}
\newcommand{\mes}{\textrm{mes}}
\newcommand{\I}{\mathbb{I}}
\newcommand{\vr}[2]{\mathcal{VR}(#1;#2)}
\newcommand{\A}{\mathcal{A}}
\newcommand{\B}{\mathcal{B}}
\newcommand{\C}{\mathcal{C}}
\newcommand{\D}{\mathcal{D}}

\newtheorem{theorem}{Theorem}
\numberwithin{theorem}{section}

\newtheorem{lemma}[theorem]{Lemma}

\newtheorem{claim}{Claim}
\newtheorem{proposition}[theorem]{Proposition}
\newtheorem{defn}[theorem]{Definition}
\newtheorem{rmk}{Remark}
\newtheorem{conj}{Conjecture}

\newtheorem{question}[theorem]{Question}

\makeatletter
\@namedef{subjclassname@2020}{\textup{2020} Mathematics Subject Classification}
\makeatother

\begin{document}

\title{On  Vietoris--Rips complexes (with scale 3) of hypercube graphs}

\author{Samir Shukla}
\address{Indian Institute of Technology (IIT) Mandi,  India}
\email{samir@iitmandi.ac.in}


\begin{abstract}
	For a metric space $(X, d)$ and a scale parameter $r \geq 0$,  the Vietoris-Rips complex $\mathcal{VR}(X;r)$ is a simplicial complex on vertex set $X$, where a finite set $\sigma \subseteq X$ is a simplex if and only if the diameter of $\sigma$ is at most $r$. For $n \geq 1$, let $\I_n$ denote the $n$-dimensional hypercube graph.   In this paper, we show that $\mathcal{VR}(\mathbb{I}_n;r)$ has non trivial reduced homology only in dimensions $4$ and $7$. Therefore, we answer a question posed by Adamaszek and Adams recently.

	A (finite) simplicial complex $\Delta$ is $d$-collapsible if it can be reduced to the
	void complex by repeatedly removing a face of size at most $d$ that is contained in a unique
	maximal face of $\Delta$. The collapsibility number of $\Delta$ is the minimum integer $d$ such that $\Delta$ is $d$-collapsible. We show that  the collapsibility number of $\mathcal{VR}(\mathbb{I}_n;r)$ is $2^r$ for $r \in \{2, 3\}$.
\end{abstract}


\maketitle

\vspace{-0.38cm}

\section{Introduction}

For a metric space $(X, d)$ and a scale parameter $r \geq 0$,  the Vietoris-Rips complex $\vr{X}{r}$ is a simplicial complex on vertex set $X$, where a finite set $\sigma \subseteq X$ is a simplex if and only if the diameter of $\sigma$ is at most $r$, {\it i.e.}, $\vr{X}{r} = \{\sigma \subseteq X : |\sigma| < \infty \ \text{and} \ d(x, y) \leq r \ \forall \ x, y \in \sigma \}$; here $|\cdot|$ denotes the cardinality of a set. The Vietoris-Rips complex was first introduced by Vietoris \cite{Vietoris27} to define a homology theory for metric spaces and independently re-introduced by E. Rips for studying  hyperbolic groups, where it has been popularised as Rips-complex \cite{Ghys, Gromov87}. The idea behind introducing these complexes was to create a finite simplicial model for metric spaces. The Vietoris-Rips complex and its homology  have become an important 
tools in the applications of algebraic topology. In topological data analysis, it has been used 
to analyse data with persistent homology \cite{Bauer2021, Carlsson2009, Zomorodian2010, ZomorodianCarlsson2005}. These complexes  have been used heavily in computational topology, as a simplicial
model for point-cloud data \cite{Carlsson06, CarlssonIshkhanovDeSilvaZomorodian2008, CarlssonZomorodian2005, DesilvaCarlsson2004} and as simplicial completions of
communication links in sensor networks \cite{DesilvaGhrist2006, DesilvaGhrist2007, Muhammad2007}.
For more on these complexes, the interested
reader is referred to \cite{Adamaszek2013, AdamaszekAdam2017,  AAF2018, AdamaszekAdamFrick2016, Adamaszek2020, AFV, chambers2010vietoris,  gasparovic2018complete, lesnick2020quantifying, virk20181, virk2017approximations}.

Consider any graph $G$ as a metric space, where  the distance between any two vertices is the length of a shortest path between them.  The study of Vietoris-Rips complexes of hypercube graphs  was initiated by Adamaszek and Adams in \cite{Adamaszek2021}. These questions on
hypercubes arose from work by Kevin Emmett, Ra{\'u}l Rabad{\'a}n, and Daniel Rosenbloom related to the persistent homology formed from genetic trees, reticulate evolution, and medial
recombination \cite{emmett2016topology, emmett2015quantifying}. 

For a positive integer $n$, let  $\I_n$ denote the $n$-dimensional hypercube graph (see \Cref{defn:hypercube}). In \cite{Adamaszek2021}, Adamaszek and Adams proved that   $\vr{\I_n}{2}$ is homotopy equivalent to a wedge sum of spheres of dimension $3$. By using a computer calculation they proved the following.

\begin{proposition}\label{thm:adm} \cite{Adamaszek2021} Let $5 \leq n \leq 7$ and $0\leq i \leq 7$. Then $\tilde{H}_i(\vr{\I_n}{3};\Z) \neq 0$ if and only if $i \in \{4, 7\}$.
\end{proposition}

Further, they asked, in what homological dimensions do the Vietoris-Rips complexes $\vr{\I_n}{3}$ have nontrivial reduced homology? It is easy to check that the complexes $\vr{\I_n}{3}$ are contractible for $1 \leq n \leq 3$ and $\vr{\I_4}{3} \simeq S^7$. In this paper we prove the following.

\begin{restatable}{thm}{maintheorem}
	\label{thm:main3}
	Let $n \geq 5$. Then $\tilde{H}_i(\vr{\I_n}{3};\Z) \neq 0$ if and only if $i \in \{4, 7\}$. 
\end{restatable}

Let $\Delta$ be a (finite) simplicial complex. Let $\gamma \in \Delta$ such that  $|\gamma|\leq d$ and $\sigma \in \Delta$ is the only maximal simplex that contains $\gamma$.   An  {\it elementary $d$-collapse} of $\Delta$ is the simplicial complex $\Delta'$ obtained from $\Delta$ by
removing all those simplices $\tau $  of $\Delta$ such that
$\gamma \subseteq \tau \subseteq \sigma$, and we denote this elementary  $d$-collapse by  $
\Delta \xrightarrow{\gamma} \Delta'$.

The complex $\Delta$ is called \emph{$d$-collapsible} if there exists a sequence of elementary $d$-collapses 
\[
\Delta=\Delta_1\xrightarrow{\gamma_1} \Delta_2 \xrightarrow{\gamma_2} \cdots \xrightarrow{\gamma_{k-1}} \Delta_k=\emptyset
\]
from $\Delta$ to the void complex $\emptyset$. Clearly, if $\Delta$ is $d$-collapsible and $d < c$, then $\Delta$ is $c$-collapsible. The \emph{collapsibility number} of $\Delta$ is the minimal integer $d$ such that $\Delta $ is $d$-collapsible.

The notion of $d$-collapsibility of simplicial complexes was introduced by Wegner  \cite{Wegner1975}. In combinatorial topology it is an important problem to determine the collapsibility number or bounds for the collapsibility number  of a simplicial  complex and it  has been widely studied
(see  \cite{Civan2022, ChoiKimPark2020, Kimlew2021,Lew2018, MT09}). 
A simple consequence of $d$-collapsibility is the following: 
\begin{proposition}\cite{Wegner1975}
	\label{prop:collapsibilitysubcomplex}
	If $X$ is $d$-collapsible then it is homotopy equivalent to a simplicial complex of dimension smaller than $d$. 
\end{proposition}

Recently, Bigdeli and Faridi gave a connection between  $d$-collapsibility and the chordal complexes; and proved that $d$-collapsibility is equivalent to the chordality of the Stanley-Reisner complexes of certain ideals \cite{Bigdeli2020}. For applications  regarding Helly-type theorems, see \cite{ AhroniHolzmanJiang2019, Kalai1984, KalaiMeshulam2005}. One of the consequences of  the topological colorful Helly
theorem \cite[Theorem 2.1]{KalaiMeshulam2005} is the following.

\begin{proposition}\cite[Theorem 1.1]{Kimlew2021}
	Let X be a $d$-collapsible simplicial complex on vertex set $V$, and let $X^c = \{\sigma \subseteq  V : \sigma \notin X \}$. Then, every
	collection of $d + 1$ sets in $X^c$ has a rainbow set belonging to $X^c$.
\end{proposition}
In this paper we prove the following.

\begin{restatable}{thm}{collapsibilitytwo}
	\label{thm:collapsibility2}
	For $n \geq 3$, the collapsibility number of $\vr{\I_n}{2}$ is $4$.	
\end{restatable}

\begin{restatable}{thm}{collapsibilitythree}
	\label{thm:collapsibility}
	For $n \geq 4$, the collapsibility number  of $\vr{\I_n}{3}$ is $8$.
\end{restatable}

\vspace*{0.3cm}

\noindent{\bf Flow of the paper:} In the following Section, we list out various definitions on graph theory and simplicial complexes that are used in this paper. We also fix  a few notations, which we use throughout this paper.  In \Cref{sec:vr2}, we consider the complex $\vr{\I_n}{2}$ and prove \Cref{thm:collapsibility2}. \Cref{sec:vr3} is devoted to the complex $\vr{\I_n}{3}$ and divided into three subsections. In \Cref{subsec:maximal}, we give a characterization of maximal simplices of $\vr{\I_n}{3}$. In \Cref{subsec:collapsibility}, we prove \Cref{thm:collapsibility}. Finally in \Cref{subsec:homology}, we prove \Cref{thm:main3}. In the last section, we posed  a few conjectures and a question that arise
naturally from the work done in this paper.

\section{Preliminaries and Notations}\label{section:prelieminaries}

A  graph $G$ is a  pair $(V(G), E(G))$,  where $V(G)$  is the set of vertices
of $G$  and $E(G) \subseteq {V(G) \choose 2}$
denotes the set of edges.
If $(x, y) \in E(G)$, it is also denoted by $x \sim y$ and we say that $x$ is adjacent to $y$. 
A {\it subgraph} $H$ of $G$ is a graph with $V(H) \subseteq V(G)$ and $E(H) \subseteq E(G)$.
For a subset $U \subseteq V(G)$, the induced subgraph $G[U]$ is the subgraph whose set of vertices  is $V(G[U]) = U$
and whose set of edges is
$E(G[U]) = \{(a, b) \in E(G) \ | \ a, b \in U\}$.

A {\it graph homomorphism} from  $G$ to $H$ is a function
$\phi: V(G) \to V(H)$ such that, $(v,w) \in E(G) \implies (\phi(v),\phi(w)) \in E(H).$ A graph homomorphism $f$ is called an {\it isomorphism} if $f$ is bijective and $f^{-1}$ is also a graph homomorphism. Two graphs are called {\it isomorphic},  if there exists an isomorphism between them. If $G$ and $H$ are isomorphic, we write $G \cong H$.

Let $G$ be a graph and $v$ be a vertex of $G$. The {\it
	neighbourhood  of $v$ }is defined as $N_G(v)=\{ w \in V(G) \ |  \
(v,w) \in E(G)\}$ and the {\it closed neighbourhood} $N_G[v] = N_G(v) \cup \{v\}$. 

Let $x$ and $y$ be two distinct vertices of $G$. A {\it $xy$-path} is a sequence $x v_0 \ldots v_n y$ of vertices of $G$ such that $x \sim v_0, v_n \sim y$ and $v_i \sim v_{i+1}$ for all $0 \leq i\leq n-1$. The {\it length} of a $xy$-path is the number of edges  appearing in the path. The {\it distance} between $x$ and $y$ is the length of a shortest path (with respect to length) among all $xy$-paths and it is denoted by $d(x, y)$. Clearly, if $(x, y) \in E(G)$, then $d(x, y) = 1$. By convention, $d(v, v) = 0$ for all $v \in V(G)$.

\begin{defn}
	A (finite) abstract simplicial complex X is a collection of
	finite sets such that if $\tau \in X$ and $\sigma \subset \tau$,
	then $\sigma \in X$. 
\end{defn}
The elements  of $X$ are called {\it simplices} or {\it faces}
of $X$. The  {\it dimension} of a simplex $\sigma$ is equal to $|\sigma| - 1$. 
The dimension of an abstract  simplicial complex is the maximum of the dimensions of its simplices. The $0$-dimensional
simplices are called {\it vertices} of $X$. If $\sigma \subset \tau$, we say that
$\sigma$ is a face of $\tau$.  If a simplex has dimension $k$, it is said to
be $k${\it -dimensional} or   $k$-{\it simplex}.  The {\it boundary} of a $k$-simplex
$\sigma $ is the simplicial complex, consisting of  all  faces of $\sigma$
of dimension $\leq k-1$ and it is denoted by $Bd(\sigma).$
A simplex which is not a face of any other
simplex is called a  {\it maximal simplex} or {\it facet}. The set of maximal simplices of $X$ is denoted by $M(X)$. 

The {\it join} of two simplicial complexes $\K_1$ and $\K_2$, denoted as $\K_1 \ast \K_2$, is a simplicial complex whose simplices are disjoint union of simplices of $\K_1$ and $\K_2$. Let $\Delta^S$ denotes a $(|S|-1)$-dimensional simplex with vertex set $S$. The \emph{cone} on $\K$ with apex $a$, denoted as $C_a(\K)$, is defined as
$$ C_a (\K ) := \K \ast \Delta^{\{a\}}.$$

In this article, we consider any simplicial complex as a topological space, namely its  geometric realization. For the definition of geometric realization, we refer to  book \cite{Kozlov08} by  Kozlov. For terminologies of algebraic topology used in this article, we  refer to \cite{Hatcher02}.

Let  $X$ be a simplicial complex and $\tau, \sigma \in X$ such that
$\sigma \subsetneq \tau$ and  $\tau$ is the only maximal simplex in $X$ that contains $\sigma$.
A  {\it simplicial collapse} of $X$ is the simplicial complex $Y$ obtained from $X$ by
removing all those simplices $\gamma$  of $X$ such that
$\sigma \subseteq \gamma \subseteq \tau$. Here, $\sigma$ is called a {\it free face} of
$\tau$ and $(\sigma, \tau)$ is called a {\it collapsible pair}. We denote this collapse
by $X \searrow Y$. In particular, if 
$X \searrow Y$, then $X \simeq Y$.

\begin{defn}\label{defn:hypercube}
	For a positive integer $n$, the  {\it $n$-dimensional  Hypercube graph}, denoted by  $\I_n$, is a graph whose vertex set $V(\I_n)= \{x_1 \ldots x_n : x_i \in \{0, 1\} \ \forall \ 1 \leq i \leq n\}$ and  any two vertices $x_1 \ldots x_n$ and $y_1 \ldots y_n$ are adjacent  if and only if $\sum\limits_{i=1}^n |x_i - y_i| = 1$, i.e.,  they are differ at exactly in one position (see \Cref{Figure:hypercube}).
\end{defn}

\begin{figure} [h]
	\begin{subfigure}[]{0.4\textwidth}
		\centering
		\vspace{0.4 cm}
		\begin{tikzpicture}
			[scale=0.5, vertices/.style={draw, fill=black, circle, inner
				sep=0.5pt}]
			\node[vertices, label=left:{00}] (1) at (0,0) {};
			\node[vertices, label=left:{01}] (2) at (0,4) {};
			\node[vertices,label=right:{10}] (3) at (4,0) {};
			\node[vertices,label=right:{11}] (4) at (4,4) {};
			
			\foreach \to/\from in
			{1/2,1/3, 2/4, 4/3} \draw [-] (\to)--(\from);
			
		\end{tikzpicture}
		\vspace{0.4cm}
		\caption{$\I_2$} \label{a}
	\end{subfigure}
	\begin{subfigure}[]{0.55\textwidth}
		\centering
		\begin{tikzpicture}
			[scale=0.6, vertices/.style={draw, fill=black, circle, inner
				sep=0.5pt}]
			\node[vertices, label=left:{011}] (1) at (0,4.5) {};
			\node[vertices, label=left:{010}] (2) at (0,1.5) {};
			\node[vertices,label=right:{111}] (3) at (3,4.5) {};
			\node[vertices,label=left:{001}] (4) at (1.5,3) {};
			\node[vertices,label=right:{101}] (5) at (4.5,3) {};
			\node[vertices,label=right:{100}] (6) at (4.5,0) {};
			\node[vertices,label=left:{000}] (7) at (1.5,0) {};
			\node[vertices,label=right:{110}] (8) at (3,1.5) {};
			\foreach \to/\from in
			{1/2,1/3,1/4, 2/7, 2/8, 4/7, 4/5,3/5, 3/8, 5/6, 6/7, 6/8} \draw [-] (\to)--(\from);
			
		\end{tikzpicture}
		\caption{$\I_3$} \label{b}
	\end{subfigure}
	\caption{}\label{Figure:hypercube}
\end{figure}

We now fix a few notations, which we use throughout this paper. For a positive integer  $n$, we denote the set $\{1, \ldots, n\}$ by $[n]$. Let  $v = v_1 \ldots v_n \in V(\I_n)$. For any $i \in [n]$, we let $v(i)= v_i$. For $ \{i_1, i_2, \ldots, i_k\}\subseteq [n]$, we let 
$v^{i_1, \ldots, i_k} \in V(\I_n)$  be defined by

$$
v^{i_1, \ldots, i_k}(j)  = \begin{cases}
	\ v(j) & \text{if} \  j \notin \{i_1, \ldots, i_k\},\\
	\ \{0, 1\} \setminus \{v(j)\} & \text{if} \   j \in \{i_1, \ldots, i_k\}. \\
\end{cases}
$$
Observe that for any two vertices $v, w \in V(\I_n)$, $d(v, w) = \sum\limits_{i=1}^n |v(i)-w(i)|$ and $d(v, w) = k$ if and only if $w = v^{i_1, \ldots, i_k}$ for some $i_1, \ldots, i_k \in [n]$.  Clearly, $N_{\I_n}(v) = \{v^i : i \in [n]\}$.  For $i, j, k \in [n]$, we let $K_{v}^{i, j, k} := \{v, v^{i, j}, v^{j, k}, v^{i, k}\}$. For the simplicity of notation, we write $N(v)$ and $N[v]$ for the sets $N_{\I_n}(v)$ and $N_{\I_n}[v]$ respectively.  

\begin{rmk} The vertices of $\I_n$ can be consider as  subsets of $[n]$, where the  element $v_1 \ldots v_n \in V(\I_n)$ correspond to the set $\{i: v_i=1\}$. Hence the distance  between any two vertices of $\I_n$ is same as the cardinality of the symmetric difference of corresponding sets.
	
	Using this notations, any  $\sigma \in \vr{\I_n}{r}$ is a set where the symmetric difference  between any two elements in $\sigma$ is  at most $r$.	  
	
\end{rmk}

\section{The complex  $\vr{\I_n}{2}$} \label{sec:vr2}
In this section, we prove \Cref{thm:collapsibility2}. We first characterise the maximal simplices of $\vr{\I_n}{2}$. 

\begin{lemma} \label{thm:maximal2}
	Let  $n \geq 3$  and $\sigma $ be a maximal simplex of $\vr{\I_n}{2}$. Then one of the following is true:
	\begin{itemize}
		\item[(i)] $\sigma = N[v]$ for some $v \in V(\I_n)$.
		\item[(ii)] $\sigma = \{v, v^{i_0}, v^{j_0}, v^{i_0, j_0}\}$ for some $v \in V(\I_n)$ and $i_0, j_0 \in [n]$.
		\item[(ii)] $\sigma = K_v^{i_0, j_0, k_0}$ for some $v \in V(\I_n)$ and $i_0, j_0, k_0 \in [n]$.
		
	\end{itemize}
\end{lemma}
\begin{proof}
	We consider the following cases.

	{\it \bf Case 1.} There exists a $w \in \sigma$ such that $N(w) \cap \sigma \neq  \emptyset$. 
	
	Let us first assume that there exists a vertex $w \in \sigma$ such that $N(w) \cap \sigma = \{v\}$. Since $w \in N(v)$,  
	$w = v^{p_0}$ for some $p_0 \in [n]$. We show that $\sigma = N[v]$. Suppose there exists $l_0 \in [n]$ such that $v^{l_0} \notin \sigma$. Since $\sigma$ is maximal, there exists  $x\in\sigma$ such that $d(x, v^{l_0}) \geq 3$. Further, since $v \in \sigma$, $d(v, x) \leq 2$. For any $t \in [n]\setminus \{l_0\}$, since $d(v^{l_0}, v^{t}) = 2$, we see that $x \neq v^t$ and therefore $d(v, x) = 2$. Hence $x = v^{i, j}$ for some $i, j \in [n]$. Since $N(w) \cap \sigma = \{v\}$, $p_0 \notin \{i, j\}$. But then $d(x, w) = 3$, a contradiction.  Thus $N(v)\subseteq \sigma$. Since $v \in \sigma$, we see that $N[v] \subseteq \sigma$. 	Suppose  there exists  $y \in \sigma$ such that $y \notin N[v]$. Then $d(v, y) = 2$ and therefore $y  = v^{s, t}$ for some $s, t\in [n]$. Choose $k \in [n] \setminus \{s, t\}$. Then $d(v^{k}, y) = 3$, a contradiction as $v^{k} \in \sigma$. Hence $\sigma = N[v]$. Thus $\sigma$ is of the type $(i)$.

	Now assume that $|N(w) \cap \sigma| \geq 2$ for all $w \in \sigma$. Let $u \in \sigma$. There exists $i_0, j_0 \in [n]$ such that $u^{i_0}, u^{j_0} \in \sigma.$ Thus $\{u, u^{i_0}, u^{j_0}\} \subseteq \sigma$. Since $|N(u^{i_0}) \cap \sigma | \geq 2$ and $u^{i_0} \not\sim u^{j_0}$, there exists  $ z \in \sigma \setminus \{u\}$ such that $z\sim u^{i_0}$.  Then $z = u^{i_0, k}$ for some $k \in [n] \setminus \{i_0\}$. Since $u^{j_0} \in \sigma$, $d(u^{j_0}, y) \leq  2$, thereby implying that $k = j_0$.  Thus $\{u, u^{i_0}, u^{j_0}, u^{i_0, j_0}\} \subseteq \sigma$. 	Suppose  there exists $q \in \sigma \setminus \{u, u^{i_0}, u^{j_0}, u^{i_0, j_0}\}$. If $q \sim u$, then $q = u^{i}$ for some $i \in [n]\setminus \{i_0, j_0\}$. Here $d(q, u^{i_0, j_0}) = 3$, a contradiction. Hence $q \not\sim u$, {\it i.e.}, $d(u, q) = 2$.  Then $q = v^{j, k}$ for some $j, k \in [n]$. If $\{i_0, j_0\} \cap \{j, k\} = \emptyset$, then $d(u^{i_0,j_0}, q) = 4$, a contradiction. Hence $\{i_0, j_0\} \cap \{j, k\} \neq \emptyset$. Without loss of generality we assume that $i_0 \in \{j, k\}$. In this case $d(q, u^{j_0}) = 3$, a contradiction. Thus $\sigma = \{u, u^{i_0}, u^{j_0}, u^{i_0, j_0}\} $. Hence $\sigma$ is of the type $(ii)$.
	
	{\it \bf Case 2.} $N(v) \cap \sigma = \emptyset$ for all $v \in \sigma$. 
	
	Let $v \in \sigma$. Clearly, $\{v\}$ is not a maximal simplex and therefore there exists  $x \in \sigma,x \neq v$. Since $N(v) \cap \sigma = \emptyset$ and $d(v, x) \leq 2$, we see that $d(v,x ) = 2$. There exist   $i_0, j_0 \in [n]$ such that $x= v^{i_0, j_0}$. Hence $\{v, v^{i_0, j_0}\} \subseteq \sigma$. For any $t \in [n] \setminus \{i_0, j_0\}$, since  $d(v^{i_0,t }, v) = 2 = d(v^{i_0, t}, v^{i_0, j_0})$, we see that  $\{v, v^{i_0, j_0}, v^{i_0, t}\} \in\vr{\I_n}{2} $. Thus $\{v, v^{i_0, j_0}\}$ is  not a maximal simplex and therefore there exists $y \in \sigma \setminus \{v, v^{i_0, j_0}\}$. Clearly, $d(v, y) = 2$.  There exist $ i, j \in [n]$ such that $y = v^{i, j}$. If $\{i, j\} \cap \{i_0, j_0\} = \emptyset $, then $d(y, v^{i_0, j_0}) \geq 3$, a contradiction. Hence $\{i, j\} \cap \{i_0, j_0\} \neq \emptyset$. Without loss of generality assume that $i = i_0$. Thus $\{v, v^{i_0, j_0}, v^{i_0, j}\} \subseteq \sigma.$  Since $N(v) \cap \sigma = \emptyset$, $v^{i_0} \notin \sigma$. Further, since $\sigma$ is maximal, there exists  $z \in \sigma$ such that $d(z, v^{i_0}) \geq 3$. Clearly $d(v, z) = 2$  and therefore $z = v^{k,l}$ for some $k, l \in [n]$. Since $d(z, v^{i_0}) \geq 3, i_0 \notin \{k, l\}$. Using the fact that $d(z, v^{i_0, j_0}) = 2 = d(z, v^{i_0, j})$, we conclude that $\{k, l\} = \{j_0, j\}$. Thus $\{v, v^{i_0, j_0}, v^{i_0, j}, v^{j_0, j}\} \subseteq \sigma.$  Suppose there exists a $w \in \sigma $ such that $w \notin \{v, v^{i_0, j_0}, v^{i_0, j}, v^{j_0, j}\}$.  Here, $d(v, w) = 2$ and therefore  $w = v^{s, t}$ for some $s, t \in [n]$. Since $d(w, v^{i_0, j_0}) = 2$, $\{i_0, j_0\} \cap \{s, t\} \neq \emptyset$. Further, $d(w, v^{i_0, j}) = 2$ implies  that $\{i_0, j\} \cap \{s, t\} \neq \emptyset$ and $d(w, v^{j_0, j}) = 2$ implies  that  $\{ j_0, j\} \cap \{s, t\} \neq \emptyset$, which is not possible. Hence $\sigma = \{v, v^{i_0, j_0}, v^{i_0, j}, v^{j_0, j}\} = K_v^{i_0, j_0, j}$. Thus  $\sigma$ is of the type $(iii)$.
\end{proof}

\begin{lemma}\label{lem:3neighbour2}
	
	Let $n \geq 3$ and $\sigma \in \vr{\I_n}{2}$ be a maximal simplex. If for some $v$, $|N(v) \cap \sigma| \geq 3$, then either $\sigma = K_v^{i_0, j_0, k_0}$ for some $i_0, j_0, k_0 \in [n] $ or $ N(v)\subseteq \sigma$.
\end{lemma}
\begin{proof}
	Let  $|N(v) \cap \sigma| \geq 3$.  If $n = 3$, then $|N(v)| = 3$ and therefore $N(v) \subseteq \sigma$. So assume that $n \geq 4$. Suppose $N(v) \not\subseteq \sigma$. Then there exists $l_0\in [n]$ such that 
	$v^{l_0} \notin \sigma$. Since $|N(v) \cap \sigma | \geq 3 $, there exist $i_0, j_0, k_0\in [n]$ such that $\{v^{i_0}, v^{j_0}, v^{k_0}\} \subseteq \sigma$.   Clearly, $l_0 \notin \{ i_0, j_0, k_0\}$. Since $v^{l_0} \notin \sigma$ and $\sigma$ is a maximal simplex,  there exists $x \in \sigma $ such that $d(x, v^{l_0}) \geq 3$. Observe that for any vertex  $u$, if $d(v, u) =1 $, then $d(u, v^{l_0}) \leq 2$. Hence $d(v, x) \geq  2$. 	If $d(v, x) \geq 4$, then $d(v^{i_0}, x) \geq 3$, a contradiction as $v^{i_0} \in \sigma$.  Hence $d(v, x) \leq 3$. If  $d(v, x) = 3$, then $x =v^{i, j, k}$ for some $i, j, k \in [n]$. Since $d(v^{i_0}, x) \leq 2$, $i_0\in \{i, j, k\}$. Similarly $j_0, k_0 \in \{i, j, k\}$. Hence $\{i, j, k\} = \{i_0, j_0, k_0\}$. Thus $\sigma =  \{v^{i_0}, v^{j_0}, v^{k_0}, v^{i_0, j_0, k_0}\} = K_v^{i_0, j_0, k_0}$. 
	
	Suppose $d(v, x) = 2$. Here,  $x = v^{i, j}$ for some $i, j\in [n]$. If $i_0 \notin \{i, j\}$, then $d(x, v^{i_0}) =3$, a contradiction as $v^{i_0} \in \sigma$.  Hence 
	$i_0 \in \{i, j\}$. By similar argument, we can show that $j_0, k_0 \in \{i, j\}$. Hence $\{i_0, j_0, k_0\} \subseteq \{i, j\}$, which is not possible.   Thus,
	$N(v) \subseteq \sigma$. 
\end{proof}

We  now review a result,  which will play a key role in the proof of \Cref{thm:collapsibility2}.  

Let $X$ be a simplicial complex on   vertex set $[n]$ and let $\prec: \sigma_1,\ldots,\sigma_m$ be a  linear ordering  of the maximal simplices  of $X$.
Given a $\sigma \in X$,   the \textit{minimal exclusion sequence} $\mes_{\prec}(\sigma)$ is defined as follows.
Let $i$ denote the smallest index such that $\sigma \subseteq \sigma_i$.
If $i=1$, then $\mes_{\prec}(\sigma)$ is the null sequence.
If $i\ge 2$, then $\mes_{\prec}(\sigma)=(v_1,\ldots, v_{i-1})$ is a finite sequence of length $i-1$ such that
$v_1=\min (\sigma\setminus \sigma_1)$ and  for each $k\in\{2, \ldots, i-1\}$, 
\[v_k=\begin{cases}
	\min(\{v_1,\dots,v_{k-1}\}\cap (\sigma \setminus \sigma_k)) & \text{if } \{v_1,\dots,v_{k-1}\}\cap (\sigma \setminus \sigma_k)\neq\emptyset,\\
	\min (\sigma\setminus \sigma_k) & \text{otherwise.}
\end{cases} \]

Let $M_{\prec}(\sigma)$ denote the set of vertices appearing in $\mes_{\prec}(\sigma)$. Define
\[d_{\prec}(X):=\max_{\sigma \in X}|M_{\prec}(\sigma)|.\]
The following result was stated and proved in \cite[Proposition 1.3]{MT09} as a special case where
$X$ is the nerve of a finite family of sets and then generalized by Lew for arbitrary simplical complex.

\begin{proposition}\cite[Theorem 6]{Lew2018} \label{thm:minimalexclusion}
	If $\prec$ is a linear ordering of the maximal simplices   of $X$, then $X$ is $d_{\prec}(X)$-collapsible.
\end{proposition}

We are now ready to prove  main result of this section.

\begin{proof}[Proof of \Cref{thm:collapsibility2}]
	
	We must show that for $n \geq 3$, the collapsibility number of $\vr{\I_n}{2}$ is $4$. Since $\vr{\I_n}{2}$ is homotopy equivalent to a  wedge sum of spheres of dimension $3$ \cite{Adamaszek2021}, 	$\tilde{H}_3(\vr{\I_n}{2}) \neq 0$ and therefore by using \Cref{prop:collapsibilitysubcomplex} we conclude that collapsibility of $\vr{\I_n}{2}$  is $\geq 4$. It is enough to show that $\vr{\I_n}{2}$ is $4$-collapsible. 
	From \Cref{thm:maximal2}, each maximal simplex  of  $\vr{\I_n}{2}$ is of the form either $(i)$ $N[v]$  or  $(ii)$ $\{v, v^{i}, v^{j}, v^{i, j}\}$ or $(iii)$ $K_v^{i_0, j_0, k_0}$. It is easy to check that  these three sets of maximal simplices are pairwise disjoint sets. Choose a linear order $\prec_1$ on maximal simplices of the type $(i)$. Extend $\prec_1$ to a linear order $\prec$ on maximal simplices  of  $\vr{\I_n}{2}$, where maximal simplices of the type $(i)$ are ordered first, {\it i.e.}, for any two maximal simplices $\sigma_1$ and $\sigma_2$, if $\sigma_1 = N[v]$ for some $v$ and $\sigma_2$ is of the type $(ii)$ or $(iii)$, then $\sigma_1 \prec \sigma_2$. Let $\tau \in \vr{\I_n}{2}$.  Let $\sigma$ be the smallest (with respect to $\prec$) maximal simplex   of $\vr{\I_n}{2}$ such that $\tau \subseteq \sigma$.   If $\sigma \neq N[v]$ for all $v \in V(\I_n)$, then $|\sigma| = 4$ and therefore by definition $|M_{\prec} (\tau)| \leq 4 $. So, assume that  $\sigma = N[v]$ for some $v \in V(\I_n)$.  We first prove that $|M_{\prec} (\tau) \cap N(v) | \leq 3$.
	
	Let  $\mes_{\prec}(\tau) = (x_1, \ldots, x_t)$. Suppose $|M_{\prec}(\tau) \cap N(v)| \geq 4$. Let $k$ be the least integer  such that $|\{x_1, \ldots, x_k\} \cap N(v)| = 3$.  Clearly, $k < t$. Let  $\{x_1, \ldots, x_k\} \cap N(v) = \{x_{i_1}, x_{i_2}, x_{i_3} \} $. Observe that $x_k \in \{x_{i_1}, x_{i_2}, x_{i_3}\}$. We show that 
	$\{x_1, \ldots, x_{k+1}\} \cap N(v) = \{x_{i_1}, x_{i_2}, x_{i_3}\}$. Let $\gamma$ be a maximal simplex such that $\gamma \prec \sigma$. If $ \{x_1, \ldots, x_{k}\} \cap (\sigma \setminus \gamma) \neq  \emptyset $, then $x_{k+1} \in \{x_1, \ldots, x_k\}$. Hence 	$\{x_1, \ldots, x_{k+1}\} \cap N(v) = \{x_{i_1}, x_{i_2}, x_{i_3}\}$. If $ \{x_1, \ldots, x_{k}\} \cap (\sigma \setminus \gamma) = \emptyset $, then $\{x_{i_1}, x_{i_2}, x_{i_3}\} \subseteq \gamma$. From \Cref{lem:3neighbour2}, either  $N(v)\subseteq \gamma$ or $\gamma = K_v^{i_0, j_0, k_0}$. Since $\gamma \prec \sigma$, $\gamma \neq K_v^{i_0, j_0, k_0}$. Hence  $N(v)\subseteq \gamma$. Thus  $x_{k+1} \notin N(v)$, thereby implying that 	$\{x_1, \ldots, x_{k+1}\} \cap N(v) = \{x_{i_1}, x_{i_2}, x_{i_3}\}$. 
	If $k+1 = t$, then we get a contradiction to the assumption that $|M_{\prec}(\tau) \cap N(v)| \geq 4$. Inductively assume that for  all $k \leq l < t$, $\{x_1, \ldots, x_{l}\} \cap N(v)  = \{x_{i_1}, x_{i_2}, x_{i_3}\} $. By the argument similar as above we can show that   $\{x_1, \ldots, x_{t}\} \cap N(v)  = \{x_{i_1}, x_{i_2}, x_{i_3}\} $, a contradiction. Thus  $|M_{\prec} (\tau) \cap N(v) | \leq 3$. Since $\sigma = N[v]$, we conclude that $|M_{\prec}(\tau)| \leq 4$.
	
	From \Cref{thm:minimalexclusion}, $\vr{\I_n}{2}$  is $4$-collapsible.  This completes the proof.
\end{proof}

\section{The complex  $\vr{\I_n}{3}$}\label{sec:vr3}

In this section, we prove \Cref{thm:main3} and \Cref{thm:collapsibility}. This section is divided into three subsections. In the next subsection,  we characterise the maximal simplices of  $\vr{\I_n}{3}$.  In subsection \ref{subsec:collapsibility}, using the minimal exclusion sequence, we prove \Cref{thm:collapsibility}. Finally, in subsection \ref{subsec:homology}, using the Mayer-Vietoris sequence for homology, we prove \Cref{thm:main3}. 

We first fix some notations, which we use throughout this section.  For any  $n \geq 1$,  let $\Delta_n = \vr{\I_n}{3}$. We say that a simplex $\sigma \in \Delta_n$ covers all places, if for each $i \in [n]$ there exist $v, w \in \sigma$ such that $v(i)= 1$ and $w(i) = 0$. For each $i \in \{n\}$ and $\epsilon \in \{0, 1\}$, let $\I_n^{(i, \epsilon)}$ be the induced subgraph of $\I_n$ on the vertex set $\{v \in V(\I_n) : v(i) = \epsilon \}$. Observe that $\I_n^{(i, \epsilon)} \cong \I_{n-1}$.

\subsection{Maximal simplices}\label{subsec:maximal}

We give a characterisation of the maximal simplices of $\Delta_n$ in \Cref{thm:maximalsimplex}.  We first establish a few lemmas, which we need to prove \Cref{thm:maximalsimplex}.
\begin{lemma} \label{lem:neighbour}
	Let $n \geq 5$ and $\sigma \in \Delta_n$ be a maximal simplex such that $\sigma$ covers all places. Then  $N(v) \cap \sigma \neq \emptyset$ for all $v \in \sigma$.
\end{lemma}
\begin{proof}
	
	Suppose there exists  $v \in \sigma$ such that $N(v) \cap \sigma = \emptyset$.  Since  $v^1 \notin \sigma$ and $\sigma$ is  a maximal simplex, there exists  $x \in \sigma$ such that $d(x, v^1) \geq 4$. It is easy to see that if $d(x, v) \leq 2$, then $d(x, v^1) \leq 3$. Hence $d(v, x) = 3$. Here,  $x=v^{i_0,j_0, k_0}$ for some $i_0, j_0, k_0 \in [n]$.  If $1 \in \{i_0, j_0, k_0\}$, then $d(v^1, x) = 2$. Hence $1 \notin \{i_0, j_0, k_0\}$.  Since $v^{i_0} \notin \sigma$, there exists  $y \in \sigma$ such that $d(y, v^{i_0}) \geq 4$. Here $d(v, y) = 3$.  Hence $y = v^{i, j, k}$ for some $i, j, k \in [n]$.  If $|\{i_0, j_0, k_0\} \cap \{i, j, k\}| \leq1$, then $d(y, v^{i_0, j_0, k_0}) \geq 4$, which contradict the fact that $v^{i_0, j_0, k_0} \in \sigma$. If $i_0 \in \{i, j, k\}$, then $d(v^{i_0}, y) = 2$, a contradiction. So $i_0 \notin \{i, j, k\}$ and therefore  $\{i_0, j_0, k_0\} \cap \{i, j, k\}= \{j_0, k_0\}$. Hence $y = v^{j_0, k_0, l_0}$ for some $l_0 \in [n] \setminus \{i_0\}$. So, $\{v, v^{i_0, j_0, k_0}, v^{j_0, k_0, l_0}\}\subseteq \sigma$. Since $v^{j_0} \notin \sigma$, there exists  $z \in \sigma$ such that $d(z, v^{j_0}) \geq 4$. Here,  $d(v, z) = 3$.  Since $d(z, v^{i_0, j_0, k_0}) \leq 3$ and $d(z, v^{j_0, k_0, l_0}) \leq 3$, we conclude that $z = v^{i_0,k_0, l_0}$.  Further,  since $v^{k_0} \notin \sigma$, there exists  $w \in \sigma $ such that 
	$d(w, v^{k_0}) \geq 4$. Here $d(v, w) = 3$. Since $d(w, v^{i_0, j_0, k_0}) \leq 3$, $d(w, v^{j_0, k_0, l_0}) \leq 3$  and $d(w, v^{i_0, k_0, l_0}) \leq 3$, we conclude that   $w = v^{ i_0, j_0, l_0, }$. So,
	$\{v, v^{i_0, j_0, k_0}, v^{j_0, k_0, l_0}, v^{i_0, k_0, l_0}, v^{i_0,j_0,  l_0}\} \subseteq \sigma$.  Since $n \geq 5$, there exists  $p \in [n] \setminus \{i_0, j_0, k_0, l_0\}$. Observe that $v^{i_0, j_0, k_0}(p) = v^{j_0, k_0, l_0}(p) = v^{i_0, k_0, l_0}(p) = v^{i_0, j_0, l_0}(p) = v(p)$. Since $\sigma$ covers all places, 
	there exists  $u \in \sigma$
	such that $u(p) = \{0, 1\} \setminus \{v(p)\}$. Since $N(v) \cap \sigma = \emptyset$, $u \neq v^{p}$. Thus, either $d(v, u) = 2$ or $d(v, u) = 3$. If $d(v, u) = 2$, then $u = v^{p, r}$ for some $r \in [n]$. If $r \notin \{i_0, j_0, k_0\}$, then $d(u, v^{i_0,j_0, k_0}) = 4$, a contradiction. Hence $r \in \{i_0, j_0, k_0\}$. Without loss of generality we  assume that $r = i_0$. In this case $d(u, v^{j_0, k_0, l_0}) = 4$, a contradiction. Hence $d(v, u) = 3$.  Here $u = v^{p, s, t}$ for some $s, t \in [n]$. If $|\{s, t\} \cap \{i_0, j_0, k_0\}| \leq 1$, then $d(u, v^{i_0, j_0, k_0}) \geq 4$. Hence $\{s, t\} \subseteq \{i_0, j_0, k_0\}$. Without loss of generality we   assume that $\{s, t\} = \{i_0, j_0\}$. Then $d(u, v^{i_0, k_0, l_0}) = 4$, a contradiction. Thus there exists no $u \in \sigma$ such that $u(p) = \{0, 1\} \setminus \{v(p)\}$, which is a contradiction to the hypothesis that $\sigma$ covers all places. Hence $N(v) \cap \sigma \neq \emptyset$.
\end{proof}

\begin{lemma}\label{lem:fullnbd1}
	Let $n \geq 5$ and let $\sigma \in \Delta_n$ be a maximal simplex such that $\sigma$ covers all places. If there exists a $w \in \sigma$ such that $N(w) \cap \sigma = \{v\}$, then $N(v) \subseteq \sigma$.
\end{lemma}
\begin{proof}
	Since $w \in N(v)$, $w = v^s$ for some $s \in [n]$.	Without loss of generality we assume that $v = v_1 \ldots v_n$, where $v_i = 0$ for each $i \in [n]$ and $s=n$, {\it i.e.,} $w = v^n$. Suppose $N(v) \not\subseteq \sigma$. There exists  $l_0 \in [n]$ such that $v^{l_0} \notin \sigma$. Clearly, $l_0 \neq n$. Since $\sigma $ covers all places, there exists  $x \in \sigma$ such that $x(l_0) = 1$. Further, since  $x \neq v^{l_0}$, $d(x, v) \geq 2$.  Thus, either $d(x, v) = 3$ or $d(x, v) = 2$. We consider the following two cases:
	
	{\it Case 1.} $d(x, v) = 3$.
	
	Here, $x= v^{l_0, i_0, j_0}$ for some $i_0, j_0 \in [n]$. If $n \notin \{i_0, j_0\}$, then $d(x, v^n) = 4$. Since $v^n \in \sigma$, $d(x, v^n) \leq 3$ and thereby implying that $n \in \{i_0, j_0\}$. Without loss of generality we assume that $n =j_0$, {\it i.e.}, $x = v^{l_0,  i_0, n}$.
	
	From \Cref{lem:neighbour}, 	there exists  $y \in \sigma $ such that $y \sim x$. Clearly $y \neq v, v^n$.  
	There exists $j \in [n]$ such that  $y = x^j$. If $j \notin \{l_0, i_0, n\}$, then $d(y, v) \geq 4$. Hence   $j \in \{l_0, i_0, n\}$ and thereby implying that  $d(y, v) = 2$. If $j \neq n$, then $y = v^{l_0, n}$ or $y = v^{i_0, n}$. In both the cases $y \sim w = v^n$, which is not possible since $N(w) \cap \sigma = \{v\}$. Hence $j = n$ and $y = v^{l_0, i_0}$.	So, 
	$\{v, v^n, v^{l_0, i_0, n}, v^{l_0, i_0}\} \subseteq \sigma$.
	
	Since $v^{i_0, n} \sim v^n= w$ and $N(w) \cap \sigma = \{v\}$, we see that  $v^{i_0, n} \notin \sigma$. Further, since $\sigma$ is a maximal simplex,   there exists  $z \in \sigma$ such that $d(z, v^{i_0, n}) \geq 4$.
	Observe that for any vertex $t$, if $t \sim v$, then $d (t, v^{i_0, n}) \leq 3$ and therefore we see that $z \not\sim v$. Since $z, v \in \sigma, d(z, v) \leq 3$.  Thus, either $d(z, v) = 3$ or $d(z, v) = 2$. If $d(z, v) = 3$, then   $z = v^{i, j, k}$ for some $i, j,k \in [n]$.  Observe that  if  $n \notin \{i, j, k\}$, then $d(z, v^n) = 4$, a contradiction as $v^n \in \sigma$.  Hence  $n \in \{i, j, k\}$. Without loss of generality we assume that $i = n$, {\it i.e.}, $z= v^{n, j, k}$. But, then  $d(z, v^{i_0, n}) \leq 3$, which is a contradiction as $d(z, v^{i_0, n}) \geq 4$.  Thus $d(z, v) = 2$. So, $z= v^{i, j}$ for some $i, j \in [n]$. If $\{i, j\} \cap \{i_0, n\} \neq \phi$, then $d(v^{i, j}, v^{i_0, n}) \leq 3$. Hence $d(z, v^{i_0, n}) \geq 4$ implies  that $\{i, j\} \cap \{i_0, n\} = \emptyset$.  If $\{i, j\} \cap \{l_0, n, i_0\} = \emptyset$, then $d(v^{i, j}, v^{l_0,i_0, n}) \geq 4$. Since $v^{l_0, i_0, n} \in \sigma$, $\{i, j\} \cap \{l_0, i_0, n\} \neq \emptyset$. Thus, we conclude that $\{i, j\} \cap \{l_0,  i_0, n\} = \{l_0\}$. Hence $z = v^{l_0, k_0}$ for some $k_0 \neq i_0, n$. So, 
	$\{v, v^n, v^{l_0, i_0, n}, v^{l_0, i_0}, v^{l_0, k_0}\} \subseteq \sigma$. 
	
	Since $v^{l_0, n} \sim v^n$ and $N(v^n) \cap \sigma =\{v\}$, $v^{l_0, n} \notin \sigma$.  Hence there exists  $p \in \sigma $ such that $d(p , v^{l_0, n}) \geq4$.  Observe that for any $u \sim v$, $d (u, v^{l_0, n}) \leq 3$ and therefore $p \not\sim v$. Thus, $d(p, v) \geq 2$.  Suppose  $d(p, v) = 3$.  Then   $p = v^{i, j, k}$ for some $i, j,k \in [n]$.  If   $n \notin \{i, j, k\}$, then $d(p, v^n) = 4$, a contradiction as $v^n \in \sigma$.  Hence  $n \in \{i, j, k\}$. Without loss of generality we assume that $i = n$, {\it i.e.}, $p= v^{n, j, k}$. But, then  $d(p, v^{l_0, n}) \leq 3$, which is a contradiction.  Thus $d(p, v) = 2$. 	
	So, $p = v^{i, j}$ for some $i, j \in [n]$. If $\{i, j\} \cap \{ l_0, n\} \neq \phi$, then $d(v^{i, j}, v^{ l_0, n}) \leq 3$. Hence $\{i, j\} \cap \{l_0, n\} = \emptyset$.  If $\{i, j\} \cap \{l_0, i_0, n\} = \emptyset$, then $d(v^{i, j}, v^{l_0,  i_0, n}) \geq 4$. Since $v^{l_0,  i_0, n} \in \sigma$, $\{i, j\} \cap \{l_0, i_0, n\} \neq \emptyset$. Hence $\{i, j\} \cap \{l_0, i_0, n\} = \{i_0\}$. Thus $p = v^{i_0, j}$ for some $j\neq n, l_0$. If $j \neq k_0$, then $d(p, v^{l_0, k_0}) =4$, which is a contradiction as  $v^{l_0, k_0} \in \sigma$. Hence $p = v^{i_0, k_0}$. So, 
	$\{v, v^n, v^{l_0, i_0, n}, v^{l_0, i_0}, v^{l_0, k_0}, v^{i_0, k_0}\} \subseteq \sigma$. 
	
	Since $n \geq 5$, there exists  $j_0 \in [n] \setminus \{l_0, i_0,k_0, n\}$. Further, since $\sigma$ covers all places, there exits $q \in \sigma$ such that $q(j_0) = 1$. Observe that $0= v(j_0) = v^{n}(j_0) = v^{l_0, n, i_0}(j_0) = v^{l_0, i_0}(j_0) = v^{l_0, k_0}(j_0) = v^{i_0, k_0}(j_0)$ and therefore 
	$q \notin  \{v, v^n, v^{l_0, i_0, n}, v^{l_0, i_0}, v^{l_0, k_0}, v^{i_0, k_0}\} $. Since $d(v^{j_0}, v^{l_0, i_0, n}) = 4$, $q \neq v^{j_0}$. Hence $d(q, v) \geq 2$. 
	\begin{itemize}
		\item[(1.1)] Suppose  $d(q, v) = 2$.
		
		Here, 	$q = v^{ j_0, i}$ for some $i \in [n]$. If $i \notin \{l_0, i_0, n\}$, then $d(q, v^{l_0, i_0, n}) \geq 4$. Hence $v^{l_0,i_0, n} \in \sigma$ implies that $i \in \{l_0, i_0, n\}$. However, since  $d(v^{ j_0, l_0}, v^{i_0, k_0}) =  4 = d(v^{j_0, i_0}, v^{l_0, k_0}) = d(v^{j_0, n}, v^{i_0, k_0}) $, we conclude that $i \notin \{l_0, i_0, n\} $.   Thus, there exist no $q \in \sigma$ such that $q(j_0) = 1$, which is a contradiction to the assumption that $\sigma$ covers all places.
		
		\item[(1.2)] Suppose $d(q, v) = 3$.
		
		Here,	$q = v^{j_0, i, j}$ for some $i , j \in [n]$. Observe that if  $|\{i, j\} \cap \{l_0, i_0, n\}| \leq  1$, then $d(q, v^{l_0, i_0, n}) \geq 4$, which is not possible since  $v^{l_0, i_0, n} \in \sigma$. Hence $\{i, j\} \subset \{l_0,  i_0, n\}$.  If $n \notin \{i, j\}$, then $d(q, v^{n}) = 4$, a contradiction as $v^n \in \sigma$. Hence $n \in \{i, j\}$. Without loss of generality we assume that $i=n$, {\it i.e.}, $q= v^{j_0, n, j}$, where $j \in \{l_0, i_0\}$.   It is easy to see that $d(v^{j_0, n, l_0}, v^{i_0, k_0} ) \geq 4$ and $d(v^{j_0, n, i_0}, v^{l_0, k_0} ) \geq 4$.  Since $v^{i_0, k_0}, v^{l_0, k_0} \in \sigma$, we see that $q \notin \{v^{j_0,  l_0, n}, v^{j_0,i_0, n }\}$. Thus, there exists no $q\in \sigma $ such that 
		$q(j_0) = 1$, which is a contradiction.
	\end{itemize}
	
	{\it \large Case 2.} $d(x, v) = 2$.
	
	Here,  $x = v^{l_0, s_0}$ for some $s_0 \in [n]$. So, $\{v, v^n, v^{l_0, s_0}\} \subseteq \sigma$.  Since $v^{l_0, n} \sim v^n =  w$ and $N(w) \cap  \sigma = \{v\}$, $v^{l_0, n} \notin \sigma$. Hence there exists  $y_0 \in \sigma $ such that 
	$d(y_0, v^{l_0, n}) \geq 4$. Observe that for any $t \sim v$, $d(t, v^{l_0, n}) \leq 3$ and therefore $y_0 \not\sim v$. Thus, $d(y_0, v) \geq 2$. 
	Suppose $d(y_0, v) = 3$. Then $y_0 = v^{i, j, k}$ for some $i, j, k \in [n]$. If $n \notin \{i, j, k\}$, then $d(y_0, v^n) \geq 4$, a contradiction as $y_0, v^n \in \sigma$. Hence $n \in \{i, j, k\}$. But, then $d(y_0, v^{l_0, n}) \leq 3$, again a contradiction. Thus, $d(y_0, v) = 2$. Here, $y_0 = v^{i, j}$ for some $i, j \in [n]$. If $\{i, j\} \cap \{l_0, n\} \neq \emptyset$, then $d(y_0, v^{l_0, n}) \leq 3$. Hence $\{i, j\} \cap \{l_0, n\} = \emptyset$. Further, if $s_0 \notin \{i, j\}$, then  $d(y_0, v^{l_0, s_0}) \geq 4$. Hence  
	$s_0 \in \{i, j\}$. Therefore $y_0 = v^{s_0, t_0}$ for some $t_0 \in [n], t_0 \neq l_0, n$. So, $\{v, v^n, v^{l_0, s_0}, v^{s_0, t_0}\} \subseteq \sigma$.

	Since $v^{s_0, n} \sim v^n$ and $N(v^n) \cap  \sigma = \{v\}$, $v^{s_0, n} \notin \sigma$.  Hence  there exists  $z_0 \in \sigma $ such that 
	$d(z_0, v^{s_0, n}) \geq 4$.  By the  argument similar as above for the $y_0$, we  see that  $d(z_0, v) = 2$.   Therefore $z_0 = v^{i, j}$ for some $i, j \in [n]$. If $\{i, j\} \cap \{s_0, n\} \neq \emptyset$, then $d(z_0, v^{s_0, n}) \leq 3$. Hence $\{i, j\} \cap \{s_0, n\} = \emptyset$. If $l_0 \notin \{i, j\}$, then $d(z_0,v^{l_0, s_0} ) \geq 4$ and if $t_0 \notin \{i, j\}$, then $d(z_0, v^{s_0, t_0}) \geq 4$. Since $v^{l_0, s_0}, v^{s_0, t_0} \in \sigma$,  we conclude that  $\{i, j\} = \{l_0, s_0\}$, {\it i.e.}, 
	$z_0 = v^{l_0, t_0}$. So,  
	$\{v, v^n, v^{l_0, s_0}, v^{s_0, t_0}, v^{l_0, t_0}\} \subseteq \sigma $. 
	
	Since $n \geq 5$, there exists  $m_0 \in [n] \setminus \{n, l_0, s_0, t_0\}$. Further, since $\sigma$ covers all places, there exists   $p_0 \in \sigma $ such that $p_0(m_0) =1$. Clearly, $u(m_0) = 0$ for all $u\in \{v, v^n, v^{l_0, s_0}, v^{s_0, t_0}, v^{l_0, t_0}\} $. Hence $p_0 \notin  \{v, v^n, v^{l_0, s_0}, v^{s_0, t_0}, v^{l_0, t_0}\} $. 
	\begin{claim} \label{claim1}$p_0 = v^{m_0}$.
	\end{claim}
	\begin{proof}[Proof of \Cref{claim1}]
		Since $p_0, v \in \sigma, d(p_0, v) \leq 3$. Suppose $d(p_0, v) = 3$. Then $p_0 = v^{m_0, i, j} $ for some $i, j \in [n]$. If $n \notin \{i, j\}$, then $d(p_0, v^n)= 4$ and therefore $n \in \{i, j\}$. Further, if $\{s_0, l_0\} \cap \{i, j\} = \emptyset$, then $d(p_0, v^{l_0, s_0}) \geq 4$, which is not possible since $v^{l_0, s_0} \in \sigma$. Hence $\{i, j\} \cap \{l_0, s_0\} \neq \emptyset$ and therefore we see that $p_0$ is either $ v^{m_0, n, s_0}$ or $v^{m_0, n, l_0}$. But  
		$d(v^{m_0, n, s_0},  v^{l_0, t_0}) = 5$ and	 $d(v^{m_0, n, l_0},  v^{s_0, t_0}) = 5$. Hence $p_0 \not\in  \{ v^{m_0, n, s_0}, v^{m_0, n, l_0} \}.$ Therefore $d(p_0, v) \leq 2$.
		
		If $d(p_0, v) = 2$, then $p_0 = v^{m_0, i}$ for some $i \in [n]$.  Since $v^{l_0, t_0}, v^{l_0, s_0} \in \sigma$ and $d(v^{m_0, s_0}, v^{l_0, t_0}) = 4= d(v^{m_0, t_0}, v^{l_0, s_0})$, we conclude that  
		$i \not \in \{s_0, t_0\}$. But, then $d(p_0, v^{s_0, t_0}) = 4$, a contradiction as $ v^{s_0, t_0} \in \sigma$. Hence $d (p_0, v) = 1$. Therefore $p_0 = v^{m_0}$. This completes the proof of \Cref{claim1}.
	\end{proof}
	
	So, 	$\{v, v^n, v^{l_0, s_0}, v^{s_0, t_0}, v^{l_0, t_0}, v^{m_0}\} \subseteq \sigma $.  Since $\sigma$ is a maximal simplex and 
	$v^{l_0} \notin \sigma $, there exists   $q_0 \in \sigma $ such that 
	$d(q_0, v^{l_0}) \geq 4$. Observe that  for any  $t$, if $d(v, t) \leq 2$, then $d(t, v^{l_0}) \leq 3$. Hence 
	$d(v, q_0) = 3$. Here, $q_0 = v^{i, j, k}$ for some $i, j, k \in [n]$. If $n \notin \{i, j, k\}$, then $d(q_0, v^{n}) \geq 4$. Hence $n \in \{i, j, k\}$. If $l_0 \in \{i, j, k\}$, then $d(q_0, v^{l_0}) \leq 3$. Hence $l_0 \notin \{i, j, k\}$. Further, if  $s_0 \notin \{i, j, k\}$, then $d(q_0, v^{l_0, s_0}) =5 $ and therefore  we see that $s_0 \in \{i, j, k\} $. Without loss of generality we assume that $i = n$ and $j =s_0$, {\it i.e.}, $q_0 = v^{n, s_0, k}$. If $t_0 \neq k$, then $d(q_0, v^{l_0, t_0}) \geq 4$, a contradiction as $v^{l_0, t_0} \in \sigma$. Hence $k = t_0$, {\it i.e.}, $q_0 = v^{n, s_0, t_0}$. But, then $d(q_0, v^{m_0}) = 4$, a contradiction. Thus,  there exists no 
	$q_0$ such that $d(q_0, v^{l_0}) \geq 4$,  a contradiction.
	
	Therefore  we conclude that $N(v) \subseteq \sigma$. This completes the proof. 
\end{proof}

Recall that for a $v \in V(\I_n)$ and $i_0, j_0, k_0 \in [n]$, $K_v^{i_0, j_0,k_0} = \{v, v^{i_0, j_0}, v^{i_0, k_0}, v^{j_0, k_0}\}$.
\begin{lemma}\label{lem:firsttypemaximal}
	Let $n \geq 5$ and let $\sigma \in \Delta_n$ be a maximal simplex such that $\sigma$ covers all places. If there exists a  $w \in \sigma$ such that $N(w) \cap \sigma = \{v\}$, then $\sigma = N(v)\cup K^{i_0,j_0,k_0}_v$ for some $i_0, j_0, k_0 \in [n]$.
\end{lemma}

\begin{proof}
	
	From \Cref{lem:fullnbd1}, $N(v) \subseteq \sigma$. Since $w \sim v$, $w = v^{l_0}$ for some $l_0 \in [n]$.
	Suppose there exists  $x \in \sigma $ such that $d(x, v) = 3$. Then $x = v^{i, j, k}$ for some $i, j, k \in [n]$. Choose $t \in [n] \setminus \{i, j, k\}$. Then $d(x, v^{t}) = 4$, a contradiction as $v^{t} \in N(v) \subseteq  \sigma$. Hence  $d(v, x) \leq 2 $ for all $x \in \sigma$.  Since $N(w) \cap \sigma  = \{v\}$, $v^{i, l_0} \notin \sigma$ for all $i \in [n], i \neq l_0$. Further, since $\sigma$ is a maximal simplex and  $v^{1, l_0} \notin \sigma$, there exists   $x_0 \in \sigma$ such that $d(x_0, v^{1, l_0}) \geq 4$. For any $p \sim v$, $d(p, v^{1, l_0}) \leq 3$ and therefore $d(x_0, v) = 2 $.  Hence $x_0 = v^{i_0, j_0} $ for some $i_0, j_0 \in [n]$. If  $\{i_0, j_0\} \cap \{1, l_0\} \neq \emptyset $, then $d(x_0, v^{1, l_0}) \leq 3$. Hence 
	$\{i_0, j_0\} \cap \{1, l_0\} = \emptyset $.  Thus $\{v, v^1, \ldots, v^n, v^{i_0, j_0}\} \subseteq \sigma$. Since $v^{i_0, l_0} \notin \sigma$, there exists $y_0 \in \sigma$ such that $d(y_0, v^{i_0, l_0}) \geq 4$. For any $q \in N(v)$, $d(q, v^{i_0, l_0}) \leq 3$ and therefore $d(y_0, v) \geq 2 $.  Since $d(x, v) \leq2$ for all $x \in \sigma$, we see that $d(y_0, v) = 2$. Hence $y_0 = v^{i, j} $ for some $i, j$. If  $\{i, j\} \cap \{i_0, j_0\} =  \emptyset $, then $d(y_0, v^{i_0, j_0}) \geq 4$. Hence $\{i, j\} \cap \{i_0, j_0\} \neq \emptyset$. If $i_0 \in \{i, j\}$, then $d(y_0, v^{i_0, l_0}) \leq 3$. Hence $i_0 \notin \{i, j\}$. Thus $y_0 = v^{j_0, k_0}$ for some $k_0 \neq i_0, l_0$. So, $\{v, v^1, \ldots, v^n, v^{i_0, j_0}, v^{j_0, k_0}\} \subseteq \sigma$.  Since  $v^{j_0, l_0} \notin \sigma$, there exists   $z_0 \in \sigma$ such that $d(z_0, v^{j_0, l_0}) \geq 4$. By an argument similar as above for  $y_0$, we  can see that  $z_0 = v^{i_0, k_0}$.   So, $\{v, v^1, \ldots, v^n, v^{i_0, j_0}, v^{j_0, k_0}, v^{i_0, k_0}\} \subseteq \sigma$.

	Suppose there exists   $p \in \sigma$ such that $p \notin 	\{v, v^1, \ldots, v^n, v^{i_0, j_0}, v^{i_0, k_0}, v^{j_0, k_0}\}$. Since $d(v, x) \leq 2$ for all $x \in \sigma$ and $p \notin N(v)$, we see that  $d(v, p) = 2$. Here, $p = v^{i, j}$ for some $i, j \in [n]$.
	Since $d(p, v^{i_0, j_0}) \leq 3$, $d(p, v^{i_0, k_0}) \leq 3$ and $d(p, v^{j_0, k_0}) \leq 3$,
	we see that $\{i, j\} \cap \{i_0, j_0\} \neq \emptyset, \{i, j\} \cap \{i_0, k_0\} \neq \emptyset$ and $ \{i, j\} \cap \{j_0, k_0\} \neq \emptyset$.
	But this is possible only if $\{i, j\} = \{i_0, j_0\}$,  or  $\{i, j\} = \{i_0, k_0\}$ or  $\{i, j\} = \{j_0, k_0\}$. Thus 	$\sigma = \{v, v^1, \ldots, v^n, v^{i_0, j_0}, v^{i_0, k_0}, v^{j_0, k_0}\} = N(v)\cup K_v^{i_0, j_0, k_0}$.
\end{proof}

\begin{lemma}\label{lem:3neighbour}
	Let $n \geq 5$ and $\sigma \in \Delta_n$ be a maximal simplex. If  $|N(w) \cap \sigma| \geq 2$ for all $w \in \sigma$, then there exists $\tilde{v} \in \sigma $ such that 
	$|N(\tilde{v}) \cap \sigma| \geq 3$.
\end{lemma}

\begin{proof}
	Let  $|N(w) \cap \sigma| \geq 2$ for all $w \in \sigma$.  If $|N(w) \cap \sigma | \geq 3$ for all $w \in \sigma$, then we are done. So assume that there exists $v \in \sigma $ such that  $|N(v) \cap \sigma | = 2$. There exist  $i_0, j_0 \in [n]$ such that $\{v, v^{i_0}, v^{j_0}\} \subseteq \sigma$.   Since $|N(v) \cap \sigma| = 2$, $v^{i} \notin \sigma$ for all $i \neq  i_0, j_0$. Choose  $p \in [n] \setminus \{i_0, j_0\}$. Since $v^{p} \notin \sigma$ and $\sigma$ is maximal, there exists  $x_0 \in \sigma$ such that $d(x_0, v^{p}) \geq 4$. Observe that for any $u \in V(\I_n)$, if $d(v, u) \leq 2$, then $d(v^{p}, u) \leq 3$. Hence $d(v, x_0) = 3$. 
	Here,  $x_0 = v^{i, j, k}$ for some  $ i, j, k\in [n]$. If $i_0 \notin \{i, j, k\}$, then $d(x_0, v^{i_0}) = 4$, a contradiction as $v^{i_0} \in \sigma$. Hence $i_0 \in \{i, j, k\}$. By similar argument, $j_0 \in \{i, j, k\}$. 
	Thus $x_0 = v^{i_0, j_0, k_0}$ for some $k_0 \in [n]$. If $k_0 = p$, then $d(x_0, v^{p}) = 2$, a contradiction. Hence $k_0 \neq p$. So, $\{v, v^{i_0}, v^{j_0}, v^{i_0,j_0, k_0}\} \subseteq \sigma$.  Since 
	$v^{k_0} \notin \sigma$, there exists  $y_0 \in \sigma$ such that $d(y_0, v^{k_0}) \geq 4$. By an argument similar as above, $d(v, y_0) = 3$ and $y_0 = v^{i_0, j_0, l_0}$ for some $l_0 \in [n]$.  If $l_0 = k_0$, then $d(y_0, v^{k_0}) = 2$, a contradiction. Hence $l_0 \neq k_0$. So, $\{v, v^{i_0}, v^{j_0}, v^{i_0,j_0, k_0}, v^{i_0, j_0, l_0}\} \subseteq \sigma$.  
	Observe that $N(v^{i_0, j_0, l_0}) \cap \{v, v^{i_0}, v^{j_0}, v^{i_0,j_0, k_0}\} = \emptyset$.  Since $|N(v^{i_0, j_0, l_0}) \cap \sigma| \geq 2 $, there exists  $z_0 \in \sigma$ such that $z_0 \sim v^{i_0, j_0, l_0}$. Further, $d(z_0, v) \leq 3$  implies that $z_0 \in \{v^{i_0, j_0}, v^{i_0, l_0}, v^{j_0, l_0}\}$.  We consider the following cases.
	\begin{itemize}
		
		\item[(1)]   $z_0 = v^{i_0, j_0}$.
		
		In this case,  $\{v^{i_0}, v^{j_0}, v^{i_0, j_0, l_0}\} \subseteq N(v^{i_0, j_0} )  \cap \sigma$.  We  take 	$\tilde{v} = v^{i_0, j_0}$. 
		
		\item[(2)]   $z_0 = v^{i_0, l_0}$.
		
		In this case,  $\{v, v^{i_0},  v^{j_0},  v^{i_0,j_0, k_0},   v^{i_0, j_0, l_0}, v^{i_0, l_0}\} \subseteq \sigma$.  Since $|N(v^{i_0,j_0, k_0}) \cap \sigma| \geq 2 $ and $N(v^{i_0, j_0, k_0}) \cap \{v, v^{i_0}, v^{j_0},  v^{i_0, j_0, l_0}, v^{i_0, j_0, k_0}, v^{i_0, l_0}\}  = \emptyset $, there exists $u_0 \in \sigma$ such that $u_0 \sim v^{i_0, j_0, k_0}$.  Now $d(v, u_0) \leq 3$ implies that $u_0 \in \{v^{i_0, j_0}, v^{i_0, k_0}, v^{j_0, k_0}\}$. Since $d(v^{j_0, k_0}, v^{i_0, l_0}) = 4$ and $v^{i_0, l_0} \in \sigma$, we see that  $u_0 \neq v^{j_0, k_0}$. If $u_0 = v^{i_0, j_0}$, then 
		$\{v^{i_0}, v^{j_0}, v^{i_0, j_0, l_0}\} \subseteq N(u_0) \cap \sigma$ and we take $\tilde{v} = u_0$.
		If $u_0 = v^{i_0, k_0}$,  then $\{v, v^{i_0, l_0}, v^{i_0, k_0} \}  \subseteq N(v^{i_0}) \cap \sigma$ and  we  take $\tilde{v} = v^{i_0}$.
		
		\item[(3)]   $z_0 = v^{j_0, l_0}$.
		
		In this case,  $\{v, v^{i_0},  v^{j_0},  v^{i_0,j_0, k_0},   v^{i_0, j_0, l_0}, v^{j_0, l_0}\} \subseteq \sigma$.   Since $|N(v^{i_0}) \cap \sigma| \geq 2 $ and $N(v^{i_0}) \cap \{v, v^{i_0}, v^{j_0}, v^{i_0,j_0, k_0}, v^{i_0, j_0, l_0}, v^{j_0, l_0}\}  = \{v\} $, there exists $w_0 \in \sigma, w \neq v$ such that $w_0 \sim v^{i_0}$. Since $w_0 \neq v$, we see that  $w_0 = v^{i_0, i}$ for some $i \in [n]$. 
		If $i \notin \{j_0, l_0\}$, then $d(w_0,v^{j_0, l_0}) = 4$, a contradiction as $v^{j_0, l_0} \in \sigma$. Hence $i \in \{j_0, l_0\}$. If $i = j_0$, then $w_0= v^{i_0,j_0}$ and $\{v^{i_0}, v^{j_0}, v^{i_0, j_0, l_0}\} \subseteq  N(w_0) \cap \sigma $. We  take $\tilde{v} = w_0$. So, assume that $i = l_0$, {\it i.e.}, $w_0 = v^{i_0, l_0}$.
		
		Here $\{v, v^{i_0},  v^{j_0},  v^{i_0,j_0, k_0},   v^{i_0, j_0, l_0}, v^{j_0, l_0}, v^{i_0, l_0}\} \subseteq \sigma$. Since $|N(v^{i_0,j_0, k_0}) \cap \sigma| \geq 2 $ and $N(v^{i_0, j_0, k_0}) \cap \{v, v^{i_0}, v^{j_0}, v^{i_0,j_0, k_0}, v^{i_0, j_0, l_0}, v^{j_0, l_0}, v^{i_0, l_0}\}  = \emptyset $, there exists $q_0 \in \sigma$ such that $q_0 \sim v^{i_0, j_0, k_0}$.  Since $d(v, q_0) \leq 3$, we see that $q_0 \in \{v^{i_0, j_0}, v^{i_0, k_0}, v^{j_0, k_0}\}$. Further, since  $d(v^{j_0, k_0}, v^{i_0, l_0}) = 4$,  $q_0 \neq v^{j_0, k_0}$. If $q_0= v^{i_0, j_0}$, then 
		$\{v^{i_0}, v^{j_0}, v^{i_0, j_0, l_0}\} \subseteq N(q_0) \cap \sigma$ and  we take $\tilde{v} = q_0$.
		If $q_0 = v^{i_0, k_0}$,  then $\{v, v^{i_0, l_0}, v^{i_0, k_0} \}  \subseteq N(v^{i_0}) \cap \sigma$ and  we take $\tilde{v} = v^{i_0}$.
	\end{itemize}
	This completes the proof.
\end{proof} 

\begin{lemma}\label{lem:subsetneighbour}
	Let $n \geq 5$ and $\sigma \in \Delta_n$ be a maximal simplex such that $\sigma$ covers all places. Let $|N(w) \cap \sigma| \geq 2$ for all $w \in \sigma$. If there exists a $v \in \sigma$ such that $|N(v) \cap \sigma| \geq 3$, then    $N(v)\subseteq \sigma$.  
\end{lemma}

\begin{proof}
	Without loss of generality, we assume that $v = v_1 \ldots v_n$, where $v_i = 0$ for all $i \in [n]$. Suppose $N(v) \not\subseteq \sigma$. Then there exists  $l_0 \in [n]$ such that 
	$v^{l_0} \notin \sigma$. Since $|N(v) \cap \sigma | \geq 3 $, there exist $i_0, j_0, k_0 \in [n] \setminus \{l_0\}$ such that $\{v^{i_0}, v^{j_0}, v^{k_0}\} \subseteq \sigma$.  Further, since $\sigma$ is maximal and  $v^{l_0} \notin \sigma$, there exists $x_0 \in \sigma $ such that $d(x_0, v^{l_0}) \geq 4$. Observe that $d(v, x_0) = 3$ and therefore $x_0 = v^{i, j, k}$ for some $i, j, k \in [n]$. Since $d(x_0, v^{l_0}) \geq 4$, $l_0 \notin \{i, j, k\}$. If $i_0 \notin \{i, j, k\}$, then $d(x_0, v^{i_0}) =4$, a contradiction as $v^{i_0} \in \sigma$.  Hence  $i_0 \in \{i, j, k\}$. By similar arguments, $j_0, k_0 \in \{i_, j, k\}$ and therefore $x_0 = v^{i_0, j_0, k_0}$.
	So, $\{v, v^{i_0}, v^{j_0}, v^{k_0}, v^{i_0, j_0, k_0}\} \subseteq \sigma$.

	Observe that for any $u \in \{v, v^{i_0}, v^{j_0}, v^{k_0}, v^{i_0, j_0, k_0}\}, u(l_0) = 0 $. Since $\sigma$ covers all places, there exists $y_0 \in \sigma$  such that $y_0(l_0) = 1$. Since $v^{l_0} \notin \sigma$, $y_0 \neq v^{l_0}$. Hence $d(v, y_0) \geq 2$. Suppose $d(v, y_0) = 3$. Then $y_0 = v^{l_0, i, j}$ for some $i, j$. If $k \in \{i_0, j_0, k_0\} \setminus \{i, j\}$, then $d(y_0, v^{k}) \geq 4$, a contradiction as $v^{k} \in \sigma$. Hence $d(v, y_0) = 2$. So, $y_0 = v^{l_0, i}$ for some $i \in [n]$. If $i \notin \{i_0, j_0, k_0\}$, then $d(y_0, v^{i_0, j_0, k_0}) \geq 4$. Since $v^{i_0, j_0, k_0} \in \sigma$, we see that $i \in \{i_0, j_0, k_0\}$. Without loss of generality we  assume that $i = i_0$, {\it i.e.}, $y_0 = v^{l_0, i_0}$. So, $\{v, v^{i_0}, v^{j_0}, v^{k_0}, v^{i_0, j_0, k_0}, v^{l_0, i_0}\} \subseteq \sigma$.  
	
	Observe that $N(v^{l_0, i_0}) \cap \{v, v^{i_0}, v^{j_0}, v^{k_0}, v^{i_0, j_0, k_0}, v^{l_0, i_0}\} = \{v^{i_0}\}$. Since $|N(v^{l_0, i_0}) \cap \sigma | \geq 2$, there exists  $z_0 \in \sigma, z_0 \neq v^{i_0}$ such that 
	$z_0 \sim v^{l_0, i_0}$. Further, since $z_0 \neq v^{i_0}$ and $v^{l_0} \notin \sigma$, $z_0 = v^{l_0, i_0, i}$ for some $i \in [n]$. If $i \neq j_0$, then $d(z_0, v^{j_0}) = 4$, a contradiction as $v^{j_0} \in \sigma$. Hence $z_0 = v^{l_0, i_0, j_0}$. But then $d(z_0, v^{k_0}) = 4$, a contradiction. Hence $N(v^{l_0, i_0}) \cap \sigma = \{v^{i_0}\}$, which is a contradiction. 
	
	Thus, we conclude that  $N(v) \subseteq \sigma$.
\end{proof}

\begin{lemma}\label{lem:2ndtypemaximal}
	
	Let $n \geq 5$ and $\sigma \in \Delta_n$ be a maximal simplex such that $\sigma$ covers all places. If  $|N(w) \cap \sigma| \geq 2$ for all $w \in \sigma$, then there exist $v, w \in \sigma $ such that $v \sim w
	$ and $\sigma = N(v) \cup N(w)$. 
	
\end{lemma}

\begin{proof}
	
	Using \Cref{lem:3neighbour} and \Cref{lem:subsetneighbour}, we conclude that there exists $v \in \sigma$ such that $N(v) \subseteq \sigma$. Hence $\{v, v^1, \ldots, v^n\} \subseteq \sigma$. 
	Observe that $N(v^1) \cap \{v, v^1, \ldots, v^n\} = \{v\}$.   Since $|N(v^1) \cap \sigma | \geq 2$, there exists $x_0 \in \sigma, x_0 \neq v$ such that $x_0 \sim v^1$. Then $x_0 = v^{1, i_1}$ for some $i_1 \in [n]$. So, $\{v, v^1, \ldots, v^n, v^{1, i_1}\} \subseteq \sigma$.  Choose $i_2 \in [n] \setminus \{1, i_1\}$.  
	
	Observe that   $v^{i_2} \in \sigma$ and $N(v^{i_2}) \cap \{v, v^1, \ldots, v^n, v^{1, i_1}\} = \{v\} $. Since $|N(v^{i_2}) \cap \sigma| \geq 2$, there exists $y_0 \in \sigma, y_0 \neq v$ such that $y_0 \sim v^{i_2}$. Further, since $y_0 \neq v$, we see that 
	$y_0 = v^{i_2, i}$ for some $i \in [n]$. If $i \notin \{1, i_1\}$, then $d(y_0, v^{1, i_1}) =4$, a contradiction as $v^{1, i_1} \in \sigma$.  Hence either $y_0 = v^{i_2, 1}$ or $y_0 = v^{i_2, i_1}$. If  $y_0 = v^{i_2, 1}$, then $\{v, v^{1, i_1}, v^{i_2, 1}\} \subseteq N(v^1) \cap \sigma$. Hence from \Cref{lem:subsetneighbour}, $N(v_1) \subseteq \sigma$. Thus 
	$N(v)\cup N(v^1) \subseteq \sigma$. 
	If  $y_0 = v^{ i_2, i_1}$, then $\{v, v^{1, i_1}, v^{i_2, i_1}\} \subseteq N(v^{i_1}) \cap \sigma$ and therefore  \Cref{lem:subsetneighbour} implies that  $N(v^{i_1}) \subseteq \sigma$. Hence 
	$N(v) \cup N(v^{i_1}) \subseteq \sigma$. 
	
	Thus, we have shown that there exist vertices $v, w \in \sigma$ such that $v \sim w$ and $N(v) \cup N(w) \subseteq \sigma$. We now show that $\sigma \subseteq N(v) \cup N(w)$.	Suppose there exists $z_0 \in \sigma$ such that $z_0 \notin N(v)\cup N(w)$. Since  $w \sim v$, $w = v^{l_0}$ for some $l_0 \in [n]$. Clearly, $d(z_0, v) \geq 2$. Suppose  $d(z_0, v) = 2$. Then  $z_0 = v^{i, j}$ for some $i, j \in [n]$. If $l_0 \in \{i, j\}$, then $z_0 \sim v^{l_0}$, which is a contradiction as $z_0 \notin N(w)$. Hence $l_0 \notin \{i, j\}$. Choose $k_0 \in [n] \setminus \{l_0, i, j\}$. Since $N(v^{l_0})\subseteq \sigma$ and $v^{l_0} \sim v^{l_0,k_0}$, we see that 
	$v^{l_0, k_0} \in \sigma$. But then $d(z_0, v^{l_0, k_0})= 4$, a contradiction. Now let $d(z_0, v) = 3$. Then $z_0 = v^{i, j, k}$ for some $i, j, k$. Choose $p \in [n] \setminus \{i, j, k\}$. Then  $N(v) \subseteq \sigma$ implies that $v^{p} \in \sigma$. But $d(z_0, v^p) = 4$, a contradiction. Thus, we conclude that  $N(v) \cup N(w) = \sigma$. 
\end{proof}

We are now ready to give a characterization of maximal simplices of $\Delta_n$.
Recall that for  $i \in [n]$ and $\epsilon \in \{0, 1\}$, $\I_n^{(i, \epsilon)}$ is the induced subgraph of $\I_n$ on the vertex set $\{v \in V(\I_n) : v(i) = \epsilon \}$.

\begin{lemma}\label{thm:maximalsimplex}
	Let $n \geq 4$ and let $\sigma\in \Delta_n$ be a maximal simplex. Then dim$(\sigma) \in \{7, n+3, 2n-1 \}$. Moreover, if dim$(\sigma) \neq 7$, then either $\sigma = N(v) \cup N(w)$ for some $v \sim w$, or $\sigma = N(u) \cup K_{u}^{i, j, k}$ for some $u$ and $i, j, k \in [n]$.
\end{lemma}
\begin{proof}
	
	Proof is by induction on $n$. 	Let $n = 4$. For any two vertices $v, w \in V(\I_4)$, let $\overline{\{v, w\}}$ denote a simplicial complex on two vertices, {\it i.e.}, $\overline{\{v, w\}} \cong S^0$. Let $v = 0000$. It is easy to check that 
	$$\Delta_n = \overline{\{v, v^{1, 2, 3, 4}\}} \ast \overline{ \{v^1, v^{2,3, 4}\}}  \ast \overline{\{v^2, v^{1, 3,4}\} } \ast \overline{ \{v^3, v^{1, 2, 4}\}} \ast \overline{\{v^4, v^{1,2,3}\}} \ast \overline{\{v^{1,2}, v^{3, 4}\} }\ast \overline{\{v^{1, 3}, v^{2, 4}\}} \ast \overline{\{v^{1, 4}, v^{2, 3}\} },$$
	the join of $8$-copies of $S^{0}$.	Therefore each maximal simplex of $\Delta_4$ is of dimension $7$. So assume that $n \geq 5$.  Inductively assume  that result is true for each $\vr{\I_r}{3}$, where $4 \leq r < n$. 
	
	Let $\sigma \in \Delta_n$ be a maximal simplex.  Suppose $\sigma $ covers all places. Then from \Cref{lem:neighbour}, $|N(v) \cap \sigma| \geq 1$ for all $v \in \sigma$. If there exists a vertex $w \in \sigma$ such that $N(w) \cap \sigma = \{v\}$, then from \Cref{lem:firsttypemaximal}, $\sigma = N(v)\cup K_{v}^{i_0,j_0, k_0}$ for some $i_0, j_0, k_0 \in [n]$. Clearly dim$(\sigma) =n+3$. If  $|N(v) \cap \sigma| \geq 2$ for all $v \in \sigma$, then  from \Cref{lem:2ndtypemaximal}, there exist $v, w \in \sigma$ such that $v \sim w$ and $\sigma = N(v) \cup N(w)$. It is easy to check that dim$(\sigma) = 2n-1$.
	
	So, assume that $\sigma$ does not covers all places.  There exists $l\in [n]$ such that $v(l) = w(l)$ for all $v, w \in \sigma$. Without loss of generality we assume that $v(l) = 0$ for all $v \in \sigma$. Observe that $\sigma \in  \vr{\I_n^{(l, 0)}}{3}$.   Since $\sigma$ is a maximal simplex in $\Delta_n$, $\sigma$ is maximal in $\vr{\I_n^{(l, 0)}}{3}$. Since $\I_n^{(l, 0)} \cong \I_{n-1}$, by induction hypothesis,  either dim$(\sigma) = 7$ or;  $\sigma = N_{\I_n^{l, 0}}(v) \cup N_{\I_n^{(l, 0)}}(w)$ for some $v, w \in V(\I_n^{(l, 0)}), v \in N_{\I_n^{(l, 0)}}(w) $  or $\sigma = N_{\I_n^{(l, 0)}}(u) \cup K_u^{i, j, k}$ for some $i, j, k \in [n] \setminus \{l\}$. Suppose dim$(\sigma) \neq 7$. Then  either $\sigma = N_{\I_n^{(l, 0)}}(v) \cup N_{\I_n^{(l, 0)}}(w)$ for some $v, w \in V(\I_n^{(l, 0)}), v \in N_{\I_n^{(l, 0)}}(w) $ or $\sigma = N_{\I_n^{(l, 0)}}(v) \cup K_v^{i, j, k}$ for some $i, j, k \in [n] \setminus \{l\}$. In either case $v^l \notin \sigma$ and $\sigma \cup \{v^l\}$ is a simplex in $\Delta_n$, which contradicts the maximality of $\sigma$. Hence dim$(\sigma) = 7$. 
	This completes the proof. 
\end{proof}

\begin{rmk}In \cite{Kleitman1966}, Kleitmann proved that for $n \geq 2l+1$, the largest set family of subsets of $[n]$ with pairwise symmetric difference at most $2l$ contains  no more than 
	$\sum_{t=0}^{l} {n \choose t }$ elements. Hence it gives the maximum possible  size of a maximal simplex of $\vr{\I_n}{2}$ $( l= 1)$ in \Cref{thm:maximal2}   and the maximum possible  size of a maximal simplex  of $\vr{\I_n}{3}$ $(l = 2)$ in \Cref{thm:maximalsimplex}
\end{rmk}

\subsection{Collapsibility}\label{subsec:collapsibility}

In this section, we prove \Cref{thm:collapsibility}. We first establish a few lemmas, which we need to prove \Cref{thm:collapsibility}.

Let $X$ be a topological space and $A$ be a subspace of $X$.  Recall that a {\it retraction} of $X$ onto $A$ is a map $r: X \to A$ such that $r(a) = a$ for all $a \in A$.

\begin{lemma}\label{lem:retraction}
	Let $n > m$ and let $H$ be an $m$-dimensional cube subgraph of $\I_n$. Then there exists a retraction $r: \Delta_n \to \vr{H}{3}$.
\end{lemma}
\begin{proof}
	Observe that, there exist sequences $(i_1, \ldots, i_{n-m} )$ and $(\epsilon_1, \ldots, \epsilon_{n-m})$, where $i_1, \ldots, i_{n-m}$ $ \in [n], \epsilon_1, \ldots, \epsilon_{n-m} \in \{0,1\}$ such that $H$ is the induced subgraph of $\I_n$ on the vertex set $\{v \in V(\I_n): v(i_j) = \epsilon_j \ \forall \ 1 \leq j \leq n-m\}$.  
	Define $r_{1}: V(I_n) \to  V(I_n^{i_1, \epsilon_1})$ as follows: for $v \in V(\I_n)$ and $t \in [n]$,
	$$
	r_{1}(v)(t)  = \begin{cases}
		\ v(t)& \text{if} \ t \neq i_1,\\
		\ \epsilon_1 & \text{if} \ t =  i_1. \\
	\end{cases}
	$$
	
	We  extend the map $r_{1}$ to $\tilde{r}_{1} : \Delta_n \to \vr{\I_n^{i_1, \epsilon_1}}{3}$ by 
	$\tilde{r}_{1}(\sigma) := \{r_{1}(v) : v \in \sigma\}$ for all $\sigma \in \Delta_n$. 
	Let $\sigma \in \Delta_n$ and let $v, w \in \sigma$. Then $d(v, w) \leq 3$. If $v(i_1) = w(i_1)$, then  $r_{1}(v) = r_{1}(w)$ and therefore $d(r_{1}(v), r_{1}(w)) = d(v, w)$.  If  $v(i_1) \neq w(i_1)$, then  $d(r_{1}(v),r_{1}(w)) =  d(v, w) -1$. So,  $d(r_{1}(v), r_{1}(w)) \leq d(v, w) \leq 3 $. Thus, $\tilde{r}_{1}(\sigma) \in \vr{\I_n^{i_1, \epsilon_1}}{3}$. Hence  $\tilde{r}_{1}$  is well defined. 
	Clearly $\tilde{r}_{1}$ is onto and  for any $\sigma \in \vr{\I_n^{i_1, \epsilon_1}}{3}$, $\tilde{r}_{1} (\sigma) = \sigma$. Hence $\tilde{r}_{1}$
	is a retraction. If $m=n-1$, then  we take $r = \tilde{r}_{1}$. Suppose $m < n-1$. Let $n-m=k$.  Assume that we have a retraction $\tilde{r}_{k-1} : \Delta_n \to\vr{H_{k-1}}{3} $, where  $H_{k-1}$ is the induced subgraph of $\I_n$ on the vertex set $\{v \in V(\I_n): v(i_j) = \epsilon_j \ \forall \ 1 \leq j \leq k-1\}$.  	Define $r_{k}: V(H_{k-1}) \to  V(H)$ as follows: for $v \in V(H_{k-1})$ and $t \in [n]$,
	$$
	r_{k}(v)(t)  = \begin{cases}
		\ v(t)& \text{if} \ t \neq i_k,\\
		\ \epsilon_k & \text{if} \ t =  i_k. \\
	\end{cases}
	$$
	Extend the map $r_{k}$ to $\tilde{r}_{k} : \vr{H_{k-1}}{3} \to \vr{H}{3}$ by 
	$\tilde{r}_{k}(\sigma) := \{r_{k}(v) : v \in \sigma\}$ for all $\sigma \in \Delta_n$. Clearly,  $\tilde{r}_{k}$ is a retraction. We take $r$ as the composition of the maps $\tilde{r}_k$ and $\tilde{r}_{k-1}$. 
	This completes the proof. 
\end{proof}

\begin{lemma}\label{lem:4neighbour}
	
	Let $n \geq 5$ and $\sigma \in \Delta_n$ be a maximal simplex. If for some $v$, $|N(v) \cap \sigma| \geq 4$, then $N[v] \subseteq \sigma$.
\end{lemma}
\begin{proof}
	Let  $|N(v) \cap \sigma| \geq 4$.  Suppose $N(v) \not\subseteq \sigma$. Then there exists a $l_0\in [n]$ such that 
	$v^{l_0} \notin \sigma$. Since $|N(v) \cap \sigma | \geq 4 $, there exist $i_0, j_0, k_0, p_0 \in [n]$ such that $\{v^{i_0}, v^{j_0}, v^{k_0}, v^{p_0}\} \subseteq \sigma$.   Clearly $l_0 \notin \{ i_0, j_0, k_0, p_0\}$. Since $v^{l_0} \notin \sigma$ and $\sigma$ is a maximal simplex,  there exists $x_0 \in \sigma $ such that $d(x_0, v^{l_0}) \geq 4$. Observe that for any vertex  $u$, if $d(v, u) \leq 2$, then $d(u, v^{l_0}) \leq 3$. Hence $d(v, x_0) = 3$. Here, $x_0 = v^{i, j, k}$ for some $i, j, k \in [n]$. If $i_0 \notin \{i, j, k\}$, then $d(x_0, v^{i_0}) =4$, a contradiction as $v^{i_0} \in \sigma$.  Hence 
	$i_0 \in \{i, j, k\}$. By similar arguments, we can show that $j_0, k_0, p_0 \in \{i, j, k\}$. Hence $\{i_0, j_0, k_0, p_0\} \subseteq \{i, j, k\}$, which is not possible.   Thus,
	$N(v) \subseteq \sigma$. 
	
	Suppose $v \notin \sigma$, then there exists a vertex $y_0 \in \sigma$ such that $d(v, y_0) \geq 4$. 
	Suppose $d(v, y_0) = 4$. Let $y_0 = v^{i, j, k , l}$. Since $n \geq 5$, there exists  $t \in [n] \setminus \{i, j, k , l\}$. Then $d(y_0, v^t) \geq4$, a contradiction as $v^t \in \sigma$. Hence $d(v, y_0) \geq 5$.  But then $d(v^{i_0}, y_0) \geq 4$, again a contradiction. Hence $v \in \sigma$. Thus, $N[v] \subseteq \sigma$.
\end{proof}

\begin{lemma}\label{lem:freepair1}
	Let $n \geq 5$ and let $\sigma \in \Delta_n$ be a maximal simplex. Suppose there exists a vertex  $v$ such that 
	$\{v^{i_0, j_0}, v^{i_0, k_0}, v^{j_0, k_0}, v^{p_0}, v^{q_0}\} \subseteq \sigma$, where $p_0, q_0 \notin \{i_0, j_0, k_0\}$. Then 
	$\sigma = N(v)\cup K^{i_0,j_0,k_0}_v$.  
	
\end{lemma}

\begin{proof}
	We first show that $v^{i_0} \in \sigma$. 
	If $v^{i_0} \notin \sigma$, then there exists  $y_0 \in \sigma$ such that 
	$d(v^{i_0}, y_0) \geq 4$. Observe that $d(v, y_0) \geq 3$.  We have the following two cases. 
	\begin{itemize}
		\item[(1)] $d(v, y_0) = 3$.
		
		Here, $y_0 = v^{i,j, k}$ for some $i, j, k \in [n]$. Since $d(y_0, v^{p_0}) \leq 3$ and $d(y_0, v^{q_0}) \leq 3$, we see that $p_0, q_0 \in \{i, j, k\} $. Without loss of generality we assume that 
		$i= p_0$ and $j = q_0$, {\it i.e.}, $y_0 = v^{p_0, q_0, k}$.   Then either $d(y_0, v^{i_0, j_0}) \geq 4$, or  $d(y_0, v^{i_0, k_0}) \geq 4$, or  $d(y_0, v^{j_0, k_0}) \geq 4$, a contradiction as $v^{i_0, j_0}, v^{i_0, k_0}, v^{j_0, k_0} \in \sigma$. 
		
		\item[(2)] $d(v, y_0) \geq 4$.
		
		Observe that if $d(v, y_0) \geq 5$, then $d(y_0, v^{p_0}) \geq 4$, which is not possible, since $y_0, v^{p_0} \in \sigma$. Hence  $d(v, y_0) = 4$.	There exist $i, j, k, l \in [n]$ such that  $y_0 = v^{i,j, k, l}$.  Since $d(y_0, v^{p_0}) \leq3$ and $d(y_0, v^{q_0}) \leq 3$, we see that $p_0, q_0 \in \{i, j, k, l\}$.  Further, since $d(v^{i_0}, y_0) \geq 4 $, $i_0 \notin \{i, j, k, l\}$. If $\{j_0, k_0\} \not\subseteq \{i, j, k, l\}$, then $d(y_0, v^{j_0, k_0}) \geq 4$, a contradiction as $v^{j_0, k_0} \in \sigma$. Hence $y_0 = v^{p_0, q_0, j_0, k_0}$. But then $d(y_0, v^{i_0, j_0}) \geq 4$, a contradiction.

	\end{itemize}
	
	Hence $v^{i_0} \in \sigma$. By similar arguments  $v^{j_0}, v^{k_0} \in \sigma$. 
	Since $\{v^{i_0}, v^{j_0}, v^{k_0}, v^{p_0}, v^{q_0}\} \subseteq N(v) \cap \sigma $, from \Cref{lem:4neighbour}, $N[v] \subseteq \sigma$. Hence $ N(v)\cup K^{i_0,j_0,k_0}_v \subseteq \sigma$. From \Cref{lem:firsttypemaximal},  $ N(v)\cup K^{i_0,j_0,k_0}_v$ is a  maximal simplex and therefore  $\sigma =  N(v) \cup K^{i_0,j_0,k_0}_v$.
\end{proof}

Inspired by  \Cref{thm:maximalsimplex}, we write the set of maximal simplices of $\Delta_n$,
$M(\Delta_n) = \A_n \cup \B_n \cup \C_n$, where 
\begin{align*}
	\A_n & = \{\sigma \in M(\Delta_n) : \sigma =N(v) \cup  K_v^{i, j, k} \ \text{for some} \ v \in V(\I_n) \ \text{and} \ i, j, k \in [n]\}, \\
	\B_n & = \{\sigma \in M(\Delta_n) : \sigma = N(v) \cup N(w) \ \text{for some} \ v, w  \in V(\I_n), v \sim w\} \ \text{and}\\
	\C_n & = M(\Delta_n) \setminus  (\A_n \cup \B_n).
\end{align*}

\begin{lemma}\label{lem:removingn3}
	
	Let $n \geq 5$. Then by using a sequence of elementary $8$-collapses, $\Delta_n$ collapses to a subcomplex $\Delta_n'$, where $M(\Delta_n') = \B_n  \cup \C_n \cup \{ K_v^{i, j, k} \cup \{ v^{i}, v^{j}, v^{k}, v^l\} : v \in V(\I_n), \{i, j, k, l \} \subseteq  [n] \}$.
\end{lemma}
\begin{proof}
	
	Let $\sigma  \in \A_n$. Then $\sigma= N(v) \cup K^{i_0, j_0, k_0}_v$ for some $v \in V(\I_n)$ and $i_0, j_0, k_0 \in [n]$.  We will use the following claim.
	
	\begin{claim} \label{claim:collapsing}
		$\Delta_n$ collapses to a subcomplex $X$, where the set of maximal simplices 
		$$M(X) = M(\Delta_n) \setminus \{\sigma\} \cup \{ K_v^{i_0, j_0, k_0} \cup \{ v^{i_0}, v^{j_0}, v^{k_0},  v^i\} : i \in [n] \setminus \{i_0,j_0, k_0\} \}.$$
	\end{claim}
	\begin{proof}[Proof of \Cref{claim:collapsing}]
		
		Without loss of generality we  assume that $\{i_0, j_0, k_0\} = \{1, 2, 3\}$.  From \Cref{lem:freepair1},  ($ \{v^{1, 2}, v^{1, 3}, v^{2, 3}, v^{4}, v^{5}\}, \sigma$) is a collapsible pair.  Thus, $\sigma \searrow  \sigma \setminus \{v^{4}\}, \sigma\setminus \{v^{5}\}, \sigma\setminus \{v^{1, 2}\},  \sigma\setminus \{v^{1, 3}\},  \sigma\setminus \{v^{2, 3}\}$. Observe that 
		$\sigma\setminus \{v^{1, 2}\} \subseteq N(v) \cup N(v^{3}),  \sigma\setminus \{v^{1, 3}\} \subseteq N(v)\cup N(v^{2}), \sigma\setminus \{v^{2, 3}\} \subseteq N(v) \cup N(v^{1})$. From \Cref{thm:maximalsimplex}, for any $u  \sim w$, $N(u) \cup N(w)$ is a maximal simplex in $\Delta_n$ and therefore we see that  $\Delta_n $ collapses to a subcomplex $X_1$, where $M (X_1) = M(\Delta_n) \setminus \{\sigma\} \cup \{\sigma \setminus \{v^4\}, \sigma \setminus \{v^{5}\}\}$.
		
		Hence claim  is true if $n= 5$.  So assume that $n \geq 6$.
		From \Cref{lem:freepair1}, $( \{v^{1, 2}, v^{1, 3}, v^{2, 3}, v^{5}, v^{6}\}, $ $ \sigma)$ and $( \{v^{1, 2}, v^{1, 3}, v^{2, 3}, v^{4}, v^{6}\}, \sigma)$  are collapsible  pairs in $\Delta_n$. Hence  $(\{v^{1, 2}, v^{1, 3}, v^{2, 3}, v^{5}, v^{6}\}, \sigma \setminus \{v^4\})$ and $( \{v^{1, 2}, v^{1, 3}, v^{2, 3}, v^{4}, v^{6}\}, \sigma \setminus \{v^{5}\})$ are collapsible  pairs in $X_1$. Thus, $\sigma \setminus \{v^4\} \searrow \sigma \setminus \{v^{4}, v^{6}\}, \sigma\setminus \{v^4, v^{5}\},  \sigma\setminus \{v^4, v^{1, 2}\}, \sigma\setminus \{v^4, v^{1, 3}\}, \sigma\setminus \{v^4, v^{2, 3}\}$ and $\sigma \setminus \{v^{5}\} \searrow \sigma \setminus \{v^{5}, v^{4}\}, \sigma\setminus \{v^{5}, v^{6}\},  \sigma\setminus \{v^{5}, v^{1, 2}\}, \sigma\setminus \{v^{5}, v^{1, 3}\}, \sigma\setminus \{v^{5}, v^{2, 3}\}$. 
		
		Observe that $\sigma\setminus \{v^{4}, v^{1, 2}\}, \sigma\setminus \{v^{5}, v^{1, 2} \}\subseteq N(v) \cup N(v^{3})$,  $\sigma\setminus \{v^{4}, v^{1, 3}\},  \sigma\setminus \{v^{5}, v^{1, 3}\} \subseteq N(v) \cup N(v^{2})$,  $\sigma\setminus \{v^{4}, v^{2, 3}\},  \sigma\setminus \{v^{5}, v^{2, 3}\} \subseteq N(v) \cup N(v^{1})$. Therefore, we conclude that 
		$X_1 $   collapses to the subcomplex $X_2$, where $M (X_2) = M(\Delta_n) \setminus \{\sigma \} \cup \{\sigma \setminus \{v^4, v^{5}\}, \sigma \setminus \{v^{4}, v^{6}\}, \sigma \setminus \{v^{5}, v^{6}\}\}$. 
		
		Hence the  claim  is true if  $n= 6$.  Let $n \geq 7$  and inductively assume that $\Delta_n$ collapses to the subcomplex $X_{n-5}$, where $$M (X_{n-5}) = M(\Delta_n) \setminus \{\sigma \} \cup \{\sigma \setminus \{v^{l_1}, v^{l_2}, \ldots , v^{l_{n-5}} \}: \{l_1, l_2, \ldots, l_{n-5} \} \subseteq \{4, 5, \ldots, n-1\} \} .$$ 
		
		Let $\{i_1, i_2, \ldots, i_{n-5} \} \subseteq \{4, 5, \ldots, n-1\}$. Observe that $|\{4, 5, \ldots, n-1, n\} \setminus  \{i_1, i_2, \ldots, i_{n-5} \}| = 2$ and $n \in \{4, 5, \ldots, n-1, n\} \setminus  \{i_1, i_2, \ldots, i_{n-5} \}$.
		Let $\{n, p\} =  \{4, 5, \ldots, n-1, n\} \setminus  \{i_1, i_2, \ldots, i_{n-5} \} $.  Using \Cref{lem:freepair1}, we observe that  $\sigma \setminus \{v^{i_1}, \ldots, v^{i_{n-5}}\}$ is the only maximal simplex in $X_{n-5}$, which contains $ \{v^{1, 2}, v^{1, 3}, v^{2, 3},v^{p}, v^{n}\}$. So, $(\{v^{1, 2}, v^{1, 3}, v^{2, 3}, v^{p}, v^{n}\} , \sigma \setminus  \{v^{i_1},  \ldots , v^{i_{n-5}} \})$ is a collap-sible pair in $X_{n-5}$. 
		Therefore
		$\sigma \setminus  \{v^{i_1},  \ldots , v^{i_{n-5}} \} \searrow \sigma \setminus \{v^{i_1},  \ldots , v^{i_{n-5}}, v^{p}\}, \sigma \setminus \{v^{i_1},  \ldots , v^{i_{n-5}}, v^{n}\},  $ $\sigma \setminus \{v^{i_1},  \ldots , v^{i_{n-5}}, v^{1, 2}\},  \sigma \setminus \{v^{i_1},  \ldots , v^{i_{n-5}}, v^{1, 3}\},  \sigma \setminus \{v^{i_1},  \ldots , v^{i_{n-5}}, v^{2, 3}\}$. Clearly, 
		$ \sigma\setminus \{v^{i_1}, \ldots , v^{i_{n-5}}, $ $v^{1, 2}\} \subseteq N(v) \cup N(v^{3})$, 	$ \sigma\setminus \{v^{i_1}, \ldots , v^{i_{n-5}}, v^{1, 3}\} \subseteq N(v) \cup N(v^{2})$, 	$ \sigma\setminus \{v^{i_1}, \ldots , v^{i_{n-5}}, v^{2, 3}\} \subseteq N(v)\cup N(v^{1})$. Thus, we conclude that 
		$X_{n-5}$ collapses to a subcomplex $X_{n-5}'$, where  
		\begin{align*}
			M (X_{n-5}')  = &M(\Delta_n) \setminus \{\sigma \} \cup\{ \sigma \setminus \{v^{i_1},  \ldots , v^{i_{n-5}}, v^{p}\}, \sigma \setminus \{v^{i_1},  \ldots , v^{i_{n-5}}, v^{n}\} \} \cup \\
			& \{\sigma \setminus \{v^{l_1}, v^{l_2}, \ldots , v^{l_{n-5}} \}: \{l_1, l_2, \ldots, l_{n-5} \} \subseteq \{4, 5, \ldots, n-1\}  \setminus \{i_1, \ldots, i_{n-5}\}\}.
		\end{align*}
		By applying  an argument similar as above for each  $\{l_1, l_2, \ldots, l_{n-5}\} \subseteq \{4, 5, \ldots, n-1\}$, we get that  
		$X_{n-5}$   collapses to the subcomplex $X_{n-4}$, where 
		$$M (X_{n-4}) = M(\Delta_n) \setminus \{\sigma \} \cup \{\sigma \setminus \{v^{l_1}, v^{l_2}, \ldots , v^{l_{n-4}} \}: \{l_1, l_2, \ldots, l_{n-4} \} \subseteq \{4, 5, \ldots, n\} \}. $$
		Observe that $\{\sigma \setminus \{v^{l_1}, v^{l_2}, \ldots , v^{l_{n-4}} \}: \{l_1, l_2, \ldots, l_{n-4} \} \subseteq \{4, 5, \ldots, n\} \}  = \{K_v^{1, 2,3 } \cup \{v^{1}, v^{2}, v^{3},  v^i\} : i \in [n] \setminus \{1,2,3\} \}$.
		Thus, by induction we get that 
		$\Delta_n$ collapses to a subcomplex $X_{n-4}$, where $$M(X_{n-4}) = M(\Delta_n) \setminus \{\sigma\} \cup \{K^{1, 2, 3}_v \cup \{v^{1}, v^{2}, v^{3}, v^i\} : 4 \leq i \leq n \}.$$
		We take $X = X_{n-4}$ and this completes the  proof of \Cref{claim:collapsing}. 
	\end{proof}
	
	By applying the \Cref{claim:collapsing} for each $\tau \in \A_n$, we get that 
	$\Delta_n$ collapses to a subcomplex $\Delta_n'$, where $M(\Delta_n') = \B_n  \cup \C_n \cup \{K^{i, j, k}_v \cup \{v^{i}, v^{j}, v^{k}, v^l\} : v \in V(\I_n), \{i, j, k, l \} \subseteq  [n] \}$.
\end{proof}	

We are now ready to prove main theorem of this section. 


\begin{proof}[Proof of \Cref{thm:collapsibility}]
	We need to show that the collapsibility number of $\vr{\I_n}{3}$ is $8$. We first show that $\Delta_n$ is $8$-collapsible. It is easy to check that each maximal simplex of $\Delta_4$ is of dimension $7$. Hence $\Delta_4$ is $8$-collapsible. So assume that $n \geq 5$. 	From \Cref{lem:removingn3}, by using elementary $8$-collapses, $\Delta_n$ collapses to a subcomplex $\Delta_n'$, where $M(\Delta_n') = \B_n  \cup \C_n \cup \{K^{i, j, k}_v \cup \{v^{i}, v^{j}, v^{k}, v^l\} : v \in V(\I_n), \{i, j, k, l \} \subseteq  [n] \}$.  Let $\D_n = \{K^{i, j, k}_v \cup \{v^{i}, v^{j}, v^{k}, v^l\} : v \in V(\I_n), \{i, j, k, l \} \subseteq  [n] \}$.   Since $n \geq 5$, by using  the cardinalities of the elements of $M(\Delta_n')$, we conclude that $M(\Delta_n')= \B_n \sqcup \C_n \sqcup D_n$.

	Choose a linear order $\prec_1$ on elements of $\B_n$.   Extend $\prec_1$ to a linear order $\prec$ on maximal simplices  of  $\Delta_n'$, where elements of $\B_n$  are ordered first, {\it i.e.}, for any two $\sigma_1, \sigma_2 \in M(\Delta_n')$,  if $\sigma_1 \in \B_n$  and $\sigma_2 \in \C_n \cup \D_n$, then  $\sigma_1 \prec \sigma_2$. 
	Let $\tau \in \Delta_n'$.  Let $\sigma$ be the smallest (with respect to $\prec$) maximal simplex of $\Delta_n'$ such that $\tau \subseteq \sigma$.   If $\sigma \in \C_n \cup \D_n$, then  $|\sigma| = 8$ and thereby implying that  $|M_{\prec}(\tau)| \leq 8$.  So assume that  $\sigma \in \B_n$. There exist $v, w \in V(\I_n)$ such that $v \sim w$ and   $\sigma = N(v) \cup N(w)$. We first prove that $|M_{\prec} (\tau) \cap N(v) | \leq 4$.
	
	Let  $\mes_{\prec}(\tau) = (x_1, \ldots, x_t)$. Suppose $|M_{\prec}(\tau) \cap N(v)| \geq 5$. Let $k$ be the least integer  such that $|\{x_1, \ldots, x_k\} \cap N(v)| = 4$.  Clearly, $k < t$. Let  $\{x_1, \ldots, x_k\} \cap N(v) = \{x_{i_1}, x_{i_2}, x_{i_3}, x_{i_4} \} $. Observe that $x_k \in \{x_{i_1}, x_{i_2}, x_{i_3}, x_{i_4}\}$. Let $\gamma$ be a maximal simplex such that $\gamma \prec \sigma$.
	If $ \{x_1, \ldots, x_{k}\} \cap (\sigma \setminus \gamma) \neq  \emptyset $, then $x_{k+1} \in \{x_1, \ldots, x_k\}$. Hence 	$\{x_1, \ldots, x_{k+1}\} \cap N(v) = \{x_{i_1}, x_{i_2}, x_{i_3}, x_{i_4}\}$. If $ \{x_1, \ldots, x_{k}\} \cap (\sigma \setminus \gamma) = \emptyset $, then $\{x_{i_1}, x_{i_2}, x_{i_3}, x_{i_4}\} \subseteq \gamma$. From  \Cref{lem:4neighbour}, $N(v)\subseteq \gamma$.   Thus  $x_{k+1} \notin N(v)$, thereby implying that 	$\{x_1, \ldots, x_{k+1}\} \cap N(v) = \{x_{i_1}, x_{i_2}, x_{i_3}, x_{i_4}\}$. 
	If $k+1 = t$, then we get a contradiction to the assumption that $|M_{\prec}(\tau) \cap N(v)| \geq 5$. Inductively assume that for  all $k \leq l < t$, $\{x_1, \ldots, x_{l}\} \cap N(v)  = \{x_{i_1}, x_{i_2}, x_{i_3}, x_{i_4}\} $.  If $ \{x_1, \ldots, x_{t-1}\} \cap (\sigma \setminus \gamma) \neq  \emptyset $, then $x_{t} \in \{x_1, \ldots, x_k\}$. Hence 	$\{x_1, \ldots, x_{t}\} \cap N(v) = \{x_{i_1}, x_{i_2}, x_{i_3}, x_{i_4}\}$.    If $ \{x_1, \ldots, x_{t-1}\} \cap (\sigma \setminus \gamma) = \emptyset $, then $\{x_{i_1}, x_{i_2}, x_{i_3}, x_{i_4}\} \subseteq \gamma$. From  \Cref{lem:4neighbour}, $N(v)\subseteq \gamma$.   Thus  $x_{t} \notin N(v)$. Hence we get that $\{x_1, \ldots, x_{t}\} \cap N(v) = \{x_{i_1}, x_{i_2}, x_{i_3}, x_{i_4}\}$, which is a contradiction to the assumption that $|M_{\prec}(\tau) \cap N(v)| \geq 5$. Thus  $|M_{\prec} (\tau) \cap N(v) | \leq 4$. 
	
	By using an argument similar as above,  $|M_{\prec} (\tau) \cap N(w) | \leq 4$.	Since  $\tau \subseteq  N(v) \cup N(w)$, we see that $M_{\prec} (\tau) \leq 8$. From \Cref{thm:minimalexclusion}, $\Delta_n$ is $8$-collapsible.

	Let $X$ be the Veitoris-Rips complex of a $4$-dimensional cube subgraph  of $\I_n$. Then using \Cref{lem:retraction}, there exists a retraction  $r:\Delta_n \to X$. Since $X \cong \Delta_4$ and  $\Delta_4 \cong S^7$, we see that $\tilde{H}_7(X;\Z) \neq 0$.  Further, since $r_{\ast}: \tilde{H}_7(\Delta_n;\Z) \to \tilde{H}_7(X;\Z)$ is surjective, $\tilde{H}_7(\Delta_n;\Z) \neq 0$. Using  \Cref{prop:collapsibilitysubcomplex}, we conclude that the  collapsibility number of $\Delta_n$ is $8$.
\end{proof}
\subsection{Homology}\label{subsec:homology}
The main aim of this section is to prove \Cref{thm:main3}.  We first establish a series of lemmas, which we need to prove \Cref{thm:main3}.
We always consider the reduced homology with integer coefficients.




For $1 \leq i \leq n$ and $ \epsilon  \in  \{0, 1\}$, let $\Delta_n^{i,\epsilon} = \vr{\I_n^{(i, \epsilon)}}{3}$ and $\partial(\Delta_n) = \bigcup\limits_{i \in [n],  \epsilon \in \{0, 1\}} \Delta_n^{i,\epsilon} $.

The following lemma plays a key role in the proof of \Cref{thm:main3}.

\begin{lemma}\label{lem:homology} 
	
	Let $n \geq 5$ and let  $p \leq n-2$. Then   any $p$-cycle $c$ in $\Delta_n$ is homologous to a $p$-cycle $\tilde{c}$ in  $ \partial(\Delta_n)$. 
	
\end{lemma}
\begin{proof}

	For any chain $z = \sum a_i \sigma_i$ in $\Delta_n$, if $a_i \neq 0$, then  we say that $\sigma_i \in z$.  For a cycle $z$ in $ \Delta_n$, let $\mu(z) = \{\sigma \in z: \sigma \notin \partial(\Delta_n)\}$. Let $c$ be a $p$-cycle in $\Delta_n$. If $\mu(c) = \emptyset$, then $c$ is a $p$-cycle in $\partial(\Delta_n)$. Suppose  $\mu(c) \neq \emptyset$.   We show that 
	$c$ is homologous to a  $p$-cycle $c_1$ such that $|\mu(c_1)| < |\mu(c)|$.  Let $\sigma \in c$ such that $\sigma \notin \partial(\Delta_n)$, {\it i.e.}, $\sigma$ covers all places.  Let $\tau$ be a maximal simplex such that $\sigma \subseteq \tau$.  Using Lemmas \ref{lem:neighbour}, \ref{lem:firsttypemaximal} and \ref{lem:2ndtypemaximal}, we see that either $\tau = N(v) \cup K_v^{i_0, j_0, k_0}$ for some $v$ and $i_0, j_0, k_0 \in [n]$ or $\tau = N(v) \cup N(w)$ for some $v \sim w$.

	{\it \large Case 1.} $\tau = N(v) \cup K_v^{i_0, j_0, k_0}$.
	
	Let $i \in [n] \setminus \{i_0,j_0, k_0\}$.  Observe that  for any $ x \in \tau \setminus \{v^i\}$, $x(i) = v(i)$. Since $\sigma$ covers all places and  $x(i) = v(i)$ for all $ x \in \tau \setminus \{v^i\}$,  we see that  $v^i \in \sigma$. Thus,
	$	\{v^i : i \in [n] \setminus \{i_0, j_0, k_0\}\} \subseteq \sigma$.   Clearly, $y(t) = v(t)$ for all  $y \in \{v^i : i \in [n] \setminus \{i_0, j_0, k_0\}\}$ and $t \in \{i_0, j_0, k_0\}$. Since $\sigma$ covers all places and $v^{i_0, j_0, k_0} \notin \tau$, we conclude that $|\sigma| \geq n-1$. Since $p \leq n-2$ and $\sigma$ is $p$-dimensional, we see that $|\sigma| = n-1$.

	Suppose $v \in \sigma$.  Let $\{x_0\} = \sigma \setminus \{\{v\} \cup \{v^i : i \in [n] \setminus \{i_0, j_0, k_0\}\} \}$. For any $t \in \{i_0, j_0, k_0\}$ and $y \in \{v\} \cup \{v^i : i \in [n] \setminus \{i_0, j_0, k_0\}\}$, $y(t) = v(t)$. Hence the fact that   $\sigma$ covers all places implies that $x_0 = v^{i_0, j_0, k_0}$, which is not possible since $v^{i_0, j_0, k_0} \notin \tau$. So, $v \notin \sigma$.  Clearly, $\sigma \cup \{v\} \in \Delta_n$.
	
	
	Recall that for any simplex $\eta$, $Bd(\eta)$ denotes the simplicial boundary of $\eta$.	Let the coefficient of $\sigma$ in $c$ is $(-1)^m a_{\sigma}$ and the coefficient of $\sigma$ in $Bd(\sigma \cup \{v\})$ is $(-1)^s$.  Define a $p$-cycle  $c_1$ as follows:
	
	$$
	c_1  = \begin{cases}
		\ c - a_{\sigma} Bd(\sigma \cup \{v\})& \text{if} \  m \ \text{and} \  s \ \text{are of same parity},\\
		\ c+a_{\sigma}  Bd (\sigma \cup \{v\}) & \text{if} \  m \ \text{and} \ s  \ \text{are of opposite parity}. \\
	\end{cases}
	$$
	
	Clearly, $c$ is homologous to $c_1$. Observe that $\sigma \notin c_1$.  Let $\gamma \in c_1$ such that $\gamma \notin c$. Then observe that $v \in \gamma$ and $\gamma \subseteq \tau$. But we have seen above that  if $v \in \gamma$, then $\gamma \in \partial(\Delta_n)$, {\it i.e.}, $\gamma$ does not covers all places. Thus, we see that  $|\mu(c_1)| < |\mu(c)|$. Since $|\mu(c)|$ is finite, by repeating the above argument  finite number of times, we get a cycle $c_k$ such that $c$ is homologous to $c_k$ and $|\mu(c_k)| = 0$, {\it i.e.,} $c_k$ is a $p$-cycle in  $ \partial(\Delta_n)$. We take $\tilde{c} = c_k$.

	{\it \large Case 2.} $\tau = N(v) \cup N(w)$.
	
	For any $k \in [n]$ and $\gamma \in \Delta_n$, we say that $\gamma$ covers $k$-places, if there exist  $i_1, \ldots, i_k \in [n]$ such for each $1 \leq l \leq k$, we get $x, y \in \gamma $ such that $x(i_l) =0 $ and $y(i_l) = 1$. 
	
	Observe that if $\sigma \subseteq N(v)$ or $\sigma \subseteq N(w)$, then $\sigma$ can covers at most $p+1$-places. Since $n  > p+1$, $\sigma$ can not covers all places, a contradiction to the assumption  that $\sigma$ covers all places. Hence $N(v) \cap \sigma \neq \emptyset$ and $N(w) \cap \sigma \neq \emptyset$.		 
	
	Since $w \sim v$,   $w = v^{q}$ for some $q \in [n]$. 
	Suppose $v, w \in \sigma$. If $N(w) \cap \sigma = \{v\}$, then $\sigma = \{v, w,  v^{i_1}, \ldots, v^{i_{p-1}} \}$ for some $i_1, i_2, \ldots, i_{p-1} \in [n] \setminus \{q\}$. Observe that $\sigma$ covers only $p$-places, namely $i_1, \ldots, i_{p-1}, q$. Hence $|N(w) \cap \sigma | \geq 2$. Then 
	$\sigma = \{v, w, v^{i_1}, \ldots, v^{i_s}, v^{q, j_1}, \ldots, v^{q, j_t}\}$, for some  $i_1, \ldots, i_s, j_1 \ldots, j_t \in [n]$, where $s+t = p-1$. Here $\sigma$ can covers at most $p$ places, namely  $i_1, \ldots, i_s, j_1, \ldots, j_t, q$ and  $\sigma$ covers $p$ places only if  
	$\{i_1, \ldots, i_s\} \cap \{j_1, \ldots, j_t\} = \emptyset$. Since $p < n$, $\sigma$ does not covers all places. Hence $\{v, w\} \not\subseteq \sigma$.

	Suppose $v\in \sigma$. Then $w \notin \sigma$. If $N(w) \cap \sigma = \{v\}$, then $\sigma = \{v, v^{i_1}, \ldots, v^{i_p} \}$ for some $i_1, i_2, \ldots, i_p \in [n]$. Observe that $\sigma$ cover only $p$-places, namely $i_1, \ldots, i_p$. Hence $|N(w) \cap \sigma | \geq 2$.
	Let $\sigma = \{v, v^{i_1}, \ldots, v^{i_s}, v^{q, j_1}, \ldots, v^{q, j_t}\}$, where $i_1, \ldots, i_s, j_1 \ldots, j_t \in [n]$ and $s+t = p$. Here  $\sigma$ can covers at most $p+1$ places, namely  $i_1, \ldots, i_s, j_1, \ldots, j_t, q$ and $\sigma$ covers $p+1$ places only if  
	$\{i_1, \ldots, i_s\} \cap \{j_1, \ldots, j_t\} = \emptyset, q \notin \{i_1, \ldots, i_s\}$. Thus, we conclude that  $v \notin \sigma$. By an argument  similar as above, $w \notin \sigma$.
	
	Let the coefficient of $\sigma$ in $c$ is $(-1)^m a_{\sigma}$ and let the coefficient of $\sigma$ in $Bd(\sigma \cup \{v\})$ is $(-1)^r$.  Define a $p$-cycle  $d_1$ as follows:
	
	$$
	d_1  = \begin{cases}
		\ c - a_{\sigma} Bd(\sigma \cup \{v\})& \text{if} \  m \ \text{and} \  r \ \text{are of same parity},\\
		\ c+ a_{\sigma}  Bd(\sigma \cup \{v\}) & \text{if} \  m \ \text{and} \ r \ \text{are of opposite parity}. 
	\end{cases}
	$$
	
	Clearly, $c$ is homologous to $d_1$ and  $|\mu(d_1)| < |\mu(c)|$. Since $|\mu(c)|$ is finite, by repeating the above argument  finite number of times, we get a cycle $d_k$ such that $c$ is homologous to $d_k$ and $|\mu(d_k)| = 0$. We take $\tilde{c} = d_k$. 
	
	This completes the proof.
\end{proof}


For any $1 \leq t < n$ and two sequences  $(\epsilon_1, \ldots, \epsilon_{t})$ and $(j_1, \ldots, j_{t})$, where for each $1 \leq l \leq t$, $\epsilon_l \in \{0, 1\}$ and $j_l \in [n]$, let $\I_n^{(j_1, \epsilon_1), \ldots, (j_t, \epsilon_t)}$  denote  the induced subgraph of $\I_n$  on the vertex set $\{v_1 \ldots v_n : v_{j_l} = \epsilon_l, 1 \leq l \leq t\}$.  Observe that for any $i < n$ and an $i$-dimensional cube subgraph $H$ of $\I_n$, there exist two sequences $(\epsilon_1, \ldots, \epsilon_{n-i})$ and $(j_1, \ldots, j_{n-i})$  such that  $H = \I_n^{(j_1, \epsilon_1), \ldots, (j_{n-i}, \epsilon_{n-i})}$.

We now define a class, whose elements are the finite unions of the Vietoris-Rips complexes of cube graphs. For $n \geq 4$ and $3 \leq m \leq n$, let  $\W_n^m$ denote the collection of all  finite union $X= X_1 \cup  \ldots  \cup X_k$ ($k$ ranges over positive integers) such that $X$ satisfies the following conditions:
\begin{itemize}
	\item for each $1 \leq j \leq k$, $X_j = \vr{H_j}{3}$ for some $m$-dimensional cube subgraph $H_j$ of $\I_n$.
	
	\item if $m \neq n$, then $X \subseteq \vr{H}{3} $ for some $(m+1)$-dimensional cube subgraph $H = \I_n^{(j_1, \epsilon_1), \ldots, (j_{n-m-1}, \epsilon_{n-m-1})}$  of $\I_n$. Further, if $X \neq \partial(\vr{H}{3})$, then there exists $\lambda \in [n] \setminus \{j_1, \ldots, j_{n-m-1}\}, \epsilon \in \{0, 1\}$ such that 
	$\vr{H^{(\lambda, \epsilon)}}{3} \subseteq X$ and $\vr{H^{(\lambda, \epsilon')}}{3}\not\subseteq X$, where $\epsilon' = \{0, 1\} \setminus \{\epsilon\}$.

\end{itemize}
\begin{rmk}
	Note that $ \W_n^n = \{\Delta_n\}$ and   $\partial(\Delta_n) \in \W_n^{n-1}$. Let $X = X_1 \cup  \ldots \cup X_k \in \W_n^m$ and suppose $X \subseteq  \vr{H}{3}$, where   $H = \I_n^{(j_1, \epsilon_1), \ldots, (j_{n-m-1}, \epsilon_{n-m-1})}$.   If $k = 1$ or $X = \partial(\vr{H}{3})$, then clearly, $X$ is connected.  If $X \neq \partial(\vr{H}{3})$, then  there exists $\lambda \in [n] \setminus \{j_1,  \ldots, j_{n-m-1}\}$ and $\epsilon \in \{0, 1\}$ such that $\vr{H^{(\lambda, \epsilon)}}{3} \subseteq X$ and $\vr{H^{(\lambda, \epsilon')}}{3}\not\subseteq X$, where $\epsilon' = \{0, 1\} \setminus \{\epsilon\}$. Let $X_p = \vr{H^{(\lambda, \epsilon)}}{3}$. For each $p \neq j \in [k]$,  since $X_j \cap X_p \cong \Delta_{m-1}$ is non empty and connected, we conclude that $X$ is connected. 
\end{rmk}

For a $X =  X_1 \cup \ldots \cup X_k \in \W_n^m$, the following claim gives us a condition, when the intersection of $ \bigcup\limits_{l \neq i } X_l$ with $X_i$ ($i \in [k]$) belongs to $\W_n^{m-1}$ or the intersection of $ \bigcup\limits_{l \neq i, j } X_l$ with $X_i \cap X_j$ ($i, j \in [k]$) belongs to $\W_n^{m-2}$, which  plays a key role in the proofs of Lemmas \ref{lem:contractible}, \ref{lem:5homology} and \ref{thm:main}, while we use induction on $k$ and $m$.

\begin{claim} \label{claim:base}
	Let  $n \geq 4$ and $3 \leq m \leq n$ and let $X = X_1 \cup  \ldots  \cup X_k \in \W_n^m$.
	\begin{enumerate}
		\item[(i)] If  $k > 1$, then there  exists $q \in [k]$ such that  $ \bigcup\limits_{j \neq q} X_j \in \W_n^m$ and $X_q \cap \bigcup\limits_{j \neq q} X_j \in \W_n^{m-1}$.
			%
		\item[(ii)] If $k \geq 3$ and $m \geq 5$, then there exist $\lambda, q \in [k]$ such that for any subset $A \subseteq [k] \setminus \{q\}$  such that $\lambda \in A$ and $|A| \geq 2$,  the following are true:
		\begin{itemize}
			\item there exists $p \in A\setminus \{\lambda\}$ such that if
			$  \bigcup\limits_{i \in A \setminus \{p\}} (X_i  \cap X_p \cap X_q) \neq \emptyset$, then $\bigcup\limits_{i \in A \setminus \{p\}} (X_i  \cap X_p \cap X_q)  \in \W_n^{m-2}$.
			
			\item  $ \bigcup\limits_{j \neq q} X_j \in \W_n^m$ and $X_q \cap \bigcup\limits_{j \neq q} X_j \in \W_n^{m-1}$.
			
		\end{itemize}
		
	\end{enumerate}
\end{claim}

\begin{proof}

	Let $ X \subseteq \vr{H}{3}$, where $H = \I_n^{(j_1, \epsilon_1), \ldots, (j_{n-m-1}, \epsilon_{n-m-1})}$. 
	\begin{itemize}
		
		\item[(i)]If $X = \partial(\vr{H}{3})$, then choose $t \in [n] \setminus \{j_1, \ldots, j_{n-m-1}\}$ and $s \in \{0, 1\}$. Clearly $\vr{H^{t, s}}{3} \subseteq X$.  Without loss of generality we assume that 
		$X_1 = \vr{H^{t, s}}{3} \subseteq X$. Then $\bigcup\limits_{j\neq 1} X_j \in \W_n^m$ and 
		$X_1 \cap \bigcup\limits_{j\neq 1} X_j = \partial(X_1) \in \W_n^{m-1}$.

		Let  $X \neq \partial(\vr{H}{3})$. There exists $\mu \in [n] \setminus \{j_1,  \ldots, j_{n-m-1}\}$ and $\epsilon \in \{0, 1\}$ such that $\vr{H^{(\mu, \epsilon)}}{3} \subseteq X$ and $\vr{H^{(\mu, \epsilon')}}{3}\not\subseteq X$, where $\epsilon' = \{0, 1\} \setminus \{\epsilon\}$.
		Then $ \vr{H^{(\mu, \epsilon)}}{3} = X_p$ for some  $1 \leq p \leq k$. 
		Choose $l \in [n]\setminus \{j_1, \ldots j_{n-m-1}, \mu\} $ and $s \in \{0, 1\}$ such that  $\vr{H^{(l, s)}}{3} \subseteq X$ (such $l $ exists because $m \geq 3$). There exists $q \in [k] \setminus \{p\}$ such that  $X_q = \vr{H^{(l, s)}}{3}$. Let $Y = \bigcup\limits_{j \neq q} X_j$.  Since $X_q \cap X_p \neq \emptyset$, $X_q \cap Y \neq \emptyset$.  Clearly,   $Y \in \W_n^m$. Further, 
		$X_q \cap Y \subseteq \vr{T}{3},$ where $T = \I_n^{(j_1, \epsilon_1), \ldots, (j_{n-m-1}, \epsilon_{n-m-1}),(l, s)}$. Observe that $\vr{T^{(\mu, \epsilon)}}{3} \subseteq X_q \cap Y$ and $\vr{T^{(\mu, \epsilon')}}{3} \not \subseteq X_q \cap Y$. Hence  $X_q \cap Y \in \W_n^{m-1}$.
		
		\item[(ii)] If $X = \partial(\vr{H}{3})$. Then choose $t \in [n] \setminus \{j_1, \ldots, j_{n-m-1}\}$. Take $X_q = \vr{H^{(t, 0)}}{3}$ and $X_{\lambda} = \vr{H^{(t, 1)}}{3}$. Then clearly,  $ \bigcup\limits_{j \neq q} X_j \in \W_n^m$,  $X_q \cap \bigcup\limits_{j \neq q} X_j \in \W_n^{m-1}$. Further, since  $X_{\lambda} \cap X_q = \emptyset$, for any choice of $p$ and $A$ containing $\lambda$, we get $ \bigcup\limits_{i \in A \setminus \{p\}} (X_i  \cap X_p \cap X_q) = \emptyset$, and therefore result is true. 
		
		Assume  $X \neq \partial(\vr{H}{3})$. There exists $\mu \in [n] \setminus \{j_1,  \ldots, j_{n-m-1}\}$ and $\epsilon \in \{0, 1\}$ such that $\vr{H^{(\mu, \epsilon)}}{3} \subseteq X$ and $\vr{H^{(\mu, \epsilon')}}{3}\not\subseteq X$, where $\epsilon' = \{0, 1\} \setminus \{\epsilon\}$. Hence $X_{\lambda} = \vr{H^{(\mu, \epsilon)}}{3}$ for some $\lambda \in [k]$. Choose $t_1 \in [n] \setminus \{j_1, \ldots, j_{n-m-1}, \mu\}$ and $s_1 \in \{0, 1\}$ such that 
		$ \vr{H^{(t_1, s_1)}}{3} \subseteq X$.  There exists $q \in [k] \setminus \{\lambda\}$  such that  $X_q = \vr{H^{(t_1, s_1)}}{3}$.  Let $A \subset [k] \setminus \{q\}$ such that $\lambda \in A$ and $|A| \geq 2$. Choose $p \in A \setminus \{\lambda\}$.  There exists $t_2 \in [n] \setminus \{j_1, \ldots, j_{n-m-1}, \mu\}$ and $s_2 \in \{0, 1\}$ such that 
		$X_p = \vr{H^{(t_2, s_2)}}{3}$. Since $p \neq q, (t_1, s_1) \neq (t_2, s_2)$.  Let $Z_A = \bigcup\limits_{i \in A \setminus \{p\}} (X_i  \cap X_q)$.  If $Z_A \cap X_p \cap X_q \neq \emptyset$, then  it is easy to check that  $Z_A \cap X_p \cap X_q \subseteq \partial(\vr{H^{(t_1, s_1), (t_2, s_2)}}{3})$, $\vr{H^{(\mu, \epsilon), (t_1, s_1), (t_2, s_2)}}{3} \subseteq Z_A \cap X_p \cap X_q$ and $\vr{H^{(\mu, \epsilon'), (t_1, s_1), (t_2, s_2)}}{3} \not\subseteq Z_A \cap X_p \cap X_q$.   Hence $Z_A \cap X_p \cap X_q \in \W_n^{m-2}$. Clearly, $ \bigcup\limits_{j \neq q} X_j \in \W_n^m$ and $X_q \cap \bigcup\limits_{j \neq q} X_j \in \W_n^{m-1}$.
		
	\end{itemize}
\end{proof}

The nerve of a family of sets $(A_i)_{i \in I}$  is the simplicial complex $\mathbf{N} = \mathbf{N}(\{A_i\})$ defined on the vertex set $I$ so that a finite subset $\sigma \subseteq I$ is in $\mathbf{N}$ precisely when $\bigcap\limits_{i \in \sigma} A_i \neq \emptyset$.

\begin{proposition}\cite[Theorem 10.6]{bjorner}\label{thm:nerve}
	Let $\Delta$ be a simplicial complex and $(\Delta_i)_{i \in I}$ be a family of subcomplexes such that $\Delta = \bigcup\limits_{i \in I} \Delta_i$. Suppose every nonempty finite intersection $\Delta_{i_1} \cap \ldots \cap \Delta_{i_t}$ for $i_j \in I, t \in \mathbb{N}$ is contractible, then $\Delta$ and $\mathbf{N}(\{\Delta_i\})$ are homotopy equivalent.
	
\end{proposition}

\begin{lemma}\label{prop:upto3}
	For any $X \in \W_4^3$, 
	$\tilde{H}_j(X) = 0$ for $0 \leq j \leq 2$.
\end{lemma}
\begin{proof}
	Let $X = X_1 \cup \ldots \cup X_k$.  Observe that  each non empty intersection $X_{i_1} \cap \ldots \cap X_{i_t} $ is homeomorphic to Vietoris-Rips complex of some cube subgraph of dimension less than $4$ and therefore contractible.  From \Cref{thm:nerve}, $ X \simeq \mathbf{N}(\{X_{i}\})$. For any $i, j\in [4]$ and $\epsilon, \delta \in \{0, 1\}$, let  $\overline{\{(i, \epsilon), (j, \delta)\}}$ be a simplicial complex on vertex set $\{(i, \epsilon), (j, \delta)\}$, which is isomorphic to $S^0$. If $X = \partial(\Delta_4)$, {\it i.e.,} $X = \bigcup\limits_{i \in [4], \epsilon \in \{0, 1\}} \Delta_4^{i, \epsilon}$, then it is easy to check that 
	$$\mathbf{N}(\{X_{i}\}) \cong \overline{\{(1, 0), (1,1)\} } \ast \overline{\{(2, 0), (2,1)\} }  \ast  \overline{ \{ (3, 0), (3,1)\}} \ast\overline{ \{(4, 0), (4,1)\}}, $$
	
	the join of $4$-copies of $S^{0}$. Hence $X\simeq  \mathbf{N}(\{X_{i}\}) \simeq S^3$ and therefore $\tilde{H}_j(X) = 0$ for $0 \leq j \leq 2$.
	
	If $X \neq \bigcup\limits_{i \in [4], \epsilon \in \{0, 1\}} \Delta_4^{i, \epsilon}$, then there exists $p \in [4], \epsilon \in \{0, 1\}$ such that $\Delta_4^{p, \epsilon} \subseteq X$ but $\Delta_4^{p, \epsilon'} \not\subseteq X$, where $\{\epsilon'\}= \{0,1\} \setminus \{\epsilon\}$. It is easy to check that 
	$\mathbf{N}(\{X_{i}\})$ is a cone over the vertex $(p, \epsilon)$ and therefore it is  contractible. 
	Thus, we conclude that $\tilde{H}_j(X) = 0$ for $0 \leq j \leq 2$.
\end{proof}

\begin{lemma} \label{lem:contractible}
	Let $n \geq 5$  and $4 \leq m \leq n$. For any  $X \in \W_n^m$ and $0 \leq j \leq 3$, $\tilde{H}_j(X) = 0$.
	
\end{lemma}

\begin{proof}
	Let $X = X_1  \cup \ldots \cup X_p \in \W_n^m$.
	Proof is by induction on $m$ and $p$. Let $m = 4$. 
	If $p =1$, then $X \cong \Delta_4 \simeq S^{7}$. Hence $\tilde{H}_j(X) = 0$ for $j \leq 3$. Let $p \geq 2$.   Inductively assume that for any $l < p$ and  $i_1, \ldots, i_l \in [p]$,  if $X_{i_1} \cup \ldots \cup X_{i_l} \in \W_n^{m}$, then $\tilde{H}_j( X_{i_1} \cup \ldots \cup X_{i_l}) = 0$ for  $j \leq3$. From \Cref{claim:base} $(i)$, there exists $t \in [p]$ such that 
	$\bigcup\limits_{i \neq t} X_i \in \W_n^{m}$ and $X_t \cap \bigcup\limits_{i \neq t} X_i \in \W_n^{m-1}$. Without loss of generality assume that $t = p$ and 
	$Y = X_1 \cup \ldots \cup X_{p-1}$. Then 
	$\tilde{H}_{j}(Y)  = 0$ and $ \tilde{H}_{j}(X_p) = 0$ for $j \leq3$.  By Mayer-Vietoris sequence for homology, we have
	$$ \cdots  \longrightarrow  \tilde{H}_{j}(Y)  \oplus \tilde{H}_{j}(X_p) \longrightarrow \tilde{H}_{j}(X) \longrightarrow \tilde{H}_{j-1}(Y\cap X_p) \longrightarrow  \tilde{H}_{j-1}(Y)  \oplus \tilde{H}_{j-1}(X_p) \longrightarrow \cdots  $$
	Since $Y \cap X_p \in \W_n^3$, $\tilde{H}_j(Y \cap X_{p}) = 0$ for $j \leq 2$ by \Cref{prop:upto3}, . Thus, we conclude that 
	$\tilde{H}_j(X) = 0$ for $j \leq 3$. So for $m = 4$, result is true. Let $m \geq 5$. 
	
	{\it \large Induction hypothesis 1:} For any $4 \leq r < m$ and $ j \leq 3$, if $Y \in \W_n^r$, then   $\tilde{H}_j(Y) = 0$.
	
	Let $ 4 \leq r < m$ and $Z \in \W_n^{r+1}$. Then  $Z= Z_1  \cup \ldots \cup Z_q$ for some $q$, where each $Z_i$ is the  Vietoris-Rips complex of a $r+1$-dimensional cube subgraph of $\I_n$. We show that $\tilde{H}_j(Z) = 0$ for $j \leq 3$. Proof is by induction on $q$. If $q = 1$, then $Z \cong \Delta_{r+1}$. Since $r+1 \geq 5$, from \Cref{lem:homology}, any $i$-cycle $c$ in $Z$ is homologous to an  $i$-cycle $\tilde{c}$ in $  \partial(Z)$ for $i \leq 3$.  Hence it is enough to show that $\tilde{H}_j(\partial(Z)) = 0$ for $j \leq 3$.  Clearly,  $\partial(Z)\in \W_n^{r}$. From induction hypothesis $1$, we get that $\tilde{H}_j (\partial(Z)) = 0$ for $j\leq 3$.  So assume that $q \geq 2$.
	
	{\it \large Induction hypothesis 2:} For any $l < q, i_1, \ldots, i_l \in [q]$ and $j \leq 3$, if  $Z_{i_1} \cup  \ldots \cup  Z_{i_l} \in \W_n^{r+1}$, then $\tilde{H}_j (Z_{i_1} \cup  \ldots \cup  Z_{i_l}) = 0$.
	
	From \Cref{claim:base} $(i)$, there exists $ t \in [q]$ such that 	$\bigcup\limits_{i \neq t} Z_i \in \W_n^{r+1}$ and  $Z_t \cap \bigcup\limits_{t \neq j} Z_j \in \W_n^{r}$. Without loss of generality assume that $t = q$.	Let $U = Z_1 \cup \ldots \cup Z_{q-1}$.  Then $U \in \W_n^{r+1}$ and by induction hypothesis $2$, $\tilde{H}_j(U) = 0 $ for $0 \leq j \leq 3$.   By Mayer-Vietoris sequence for homology, we have 
	$$ \cdots  \longrightarrow  \tilde{H}_{j}(U)  \oplus \tilde{H}_{j}(Z_q) \longrightarrow \tilde{H}_{j}(Z) \longrightarrow \tilde{H}_{j-1}(U \cap Z_q) \longrightarrow  \tilde{H}_{j-1}(U)  \oplus \tilde{H}_{j-1}(Z_q) \longrightarrow \cdots  $$

	From induction hypothesis $1$,
	$\tilde{H}_j(U \cap Z_q) = 0$ for $0 \leq j \leq 3$. Therefore, we conclude that $\tilde{H}_j(Z) = 0$ for $0 \leq j \leq 3$. 
	
	Thus, the proof is complete  by induction.   
\end{proof}

\begin{lemma}\label{lem:representation}
	
	Let $n \geq m\geq6$ and  $k \geq 3$. For each $i \in [k]$, let  $X_i$ be the  Vietoris-Rips complex of some $m$-dimensional cube subgraph of $\I_n$.  Let $ Y = \bigcup\limits_{l=1}^{k-2} X_l \cap X_k $ and  $Y' = X_{k-1} \cap X_k$ such that if $Y \cap Y' \neq \emptyset$, then $Y \cap Y' \in \W_n^{m-2}$. Then for each  $x \in \tilde{H}_4(\bigcup\limits_{l=1}^{k-1} X_l \cap X_k) $, there exist $x_1 \in \tilde{H}_4(Y) $  and $x_2 \in \tilde{H}_4(Y') $ such that $x = x_1 + x_2$.
\end{lemma}

\begin{proof}
	Let $X = \bigcup\limits_{l=1}^{k-1} X_l \cap X_k$.  Then $X = Y \cup Y'$.  Let $x \in  \tilde{H}_{4}(X)$.   If $Y \cap Y' = \emptyset$, then $\tilde{H}_4(X) = \tilde{H}_4(Y) \oplus \tilde{H}_4(Y')$. Hence $x=x_1 + x_2$ for some  $x_1 \in  \tilde{H}_{4}(Y)$ and  $x_2 \in \tilde{H}_{4}(Y')$. Let $Y \cap Y' \neq \emptyset$. By Mayer-Vietoris sequence for homology, we get
	$$
	\cdots  \longrightarrow \tilde{H}_{4}(Y)  \oplus \tilde{H}_{4}(Y') \  {\overset{\psi} \longrightarrow } \tilde{H}_{4}(X) {\overset{\phi} \longrightarrow }  \tilde{H}_{3}(Y \cap Y')  \longrightarrow    \tilde{H}_{3}(Y)  \oplus \tilde{H}_{3}(Y') \longrightarrow  \cdots	
	$$
	
	Since $m-2 \geq 4$, 	from \Cref{lem:contractible}  $\tilde{H}_{3}(Y \cap Y') = 0 $. Hence 
	$\psi: \tilde{H}_{4}(Y)  \oplus \tilde{H}_{4}(Y') \longrightarrow \tilde{H}_{4}(X)$ given by $(\alpha, \beta) \mapsto \alpha + \beta$ is surjective. Thus $x=x_1 + x_2$ for some  $x_1 \in  \tilde{H}_{4}(Y)$ and  $x_2 \in \tilde{H}_{4}(Y')$.
\end{proof}

\begin{lemma}\label{lem:injective} 
	
	Let $n \geq m\geq6$ and $k \geq 2$. Let  $X = X_1 \cup \ldots \cup X_k$, where each $X_i$ is Vietoris-Rips complex of some $m$-dimensional cube subgraph of $\I_n$.  Further, assume that  if $k \geq 3$, then  for any set $A \subseteq [k-1]$ such that  $1 \in A$ and $|A| \geq 2$, there  exists $p \in A \setminus \{1\}$ such that 
	$\bigcup\limits_{i \in A \setminus \{p\}} (X_i  \cap X_p \cap X_k) \neq \emptyset$ implies 	$\bigcup\limits_{i \in A \setminus \{p\}} (X_i  \cap X_p \cap X_k) \in \W_n^{m-2}$. Then the map  $i_{\ast} : \tilde{H}_4(\bigcup\limits_{l=1}^{k-1} X_l \cap X_k) \to  \tilde{H}_4(X_k)$ induced by the inclusion $\bigcup\limits_{l=1}^{k-1} X_l \cap X_k \xhookrightarrow{} X_k$, is injective.
\end{lemma}
\begin{proof}
	Let $Y = \bigcup\limits_{l=1}^{k-1} X_l$. If $Y \cap X_k= \emptyset$, then result is vacuously true.  So assume that $Y \cap X_k \neq \emptyset$. 	If $k = 2$, then $ Y \cap X_k \cong \Delta_{m-1}$. 	From \Cref{lem:retraction},  there exists a retraction $X_k\to Y \cap X_k$ and therefore $i_{\ast} : \tilde{H}_4(Y \cap X_k) \to  \tilde{H}_4(X_k)$ is injective.

	Let $k \geq 3$ and  inductively assume that for any $1 \leq t < k-1$ and $ 1 \in \{j_1, \ldots, j_t\} \subseteq[k-1]$, the map  $i_{\ast} : \tilde{H}_4(\bigcup\limits_{l=1}^{t} X_{j_l} \cap X_{k} )\to  \tilde{H}_4(X_k)$ induced by the inclusion $\bigcup\limits_{l=1}^{t} X_{j_l} \cap X_{k} \xhookrightarrow{} X_k$, is injective.
	
	Let $B = \{j_1, \ldots , j_{t+1}\} \subseteq [k-1]$ such that $1 \in B$. Let $Z = \bigcup\limits_{l=1}^{t+1} X_{j_l}  $.   We show that the map $i_{\ast}: \tilde{H}_4 (Z \cap X_k) \to \tilde{H}_4(X_k)$ is injective.

	Let $0 \neq x\in \tilde{H}_4(Z \cap X_k)$. There exists $p \in B$, $p  \neq 1$ such that 
	$\bigcup\limits_{i \in B \setminus \{p\}} (X_i  \cap X_p \cap X_k) \neq \emptyset$ implies 	$\bigcup\limits_{i \in B \setminus \{p\}} (X_i  \cap X_p \cap X_k) \in \W_n^{m-2}$.  From \Cref{lem:representation},  there exist $x_1 \in \tilde{H}_4(\bigcup\limits_{i \in B\setminus\{p\}} X_i \cap X_{k}  )$,  $x_2 \in \tilde{H}_4 (X_{p} \cap X_k )$
	such that $x =x_1 + x_2 $.  Suppose $i_{\ast}(x) = 0$ in $\tilde{H}_4(X_k)$. Since $x \neq 0$, at least one of $x_1$ or $x_2$ is a non zero element of $ \tilde{H}_4(Z \cap X_k)$. Let $x_1 \neq 0$ in $\tilde{H}_4(Z \cap X_k)$. Then $x_1 \neq 0$ in $\tilde{H}_4(\bigcup\limits_{i \in B\setminus\{p\}} X_i \cap X_{k} )$. From induction hypothesis the map $j_{\ast} : \tilde{H}_4(\bigcup\limits_{i \in B\setminus\{p\}} X_i \cap X_{k} ) \to  \tilde{H}_4(X_k)$ induced by the inclusion $j: \bigcup\limits_{i \in B\setminus\{p\}} X_i \cap X_{k}  \to X_k$,  is injective and therefore $j_{\ast}(x_1) \neq 0$. Since $i_{\ast}(x_1) = j_{\ast}(x_1)$, we see that 
	$i_{\ast}(x_1) \neq 0$.  Further, $i_{\ast}(x) = i_{\ast}(x_1+ x_2) = x_1+x_2= 0$ implies that $x_1 = - x_2$. The injectivity of the map  $j_{\ast} : \tilde{H}_4(\bigcup\limits_{i \in B\setminus\{p\}} X_i \cap X_{k} ) \to  \tilde{H}_4(X_k)$ implies that $\tilde{H}_4(\bigcup\limits_{i \in B\setminus\{p\}} X_i \cap X_{k} ) $ is a subgroup of $\tilde{H}_4(X_k)$. Hence  $x_2 \in \tilde{H}_4(\bigcup\limits_{i \in B\setminus\{p\}} X_i \cap X_{k} )$. Therefore $x_1 + x_2 = 0$ in 
	$\tilde{H}_4(\bigcup\limits_{i \in B\setminus\{p\}} X_i \cap X_{k} )$. Hence $x = x_1+ x_2 = 0$ in $\tilde{H}_4(Z \cap X_k )$, a contradiction. By an argument similar as above, we can show that,  if $x_2 \neq 0$, then $x_1 + x_2 = 0$ in $\tilde{H}_4(Z \cap X_k)$, a  contradiction. Thus $x \neq 0$ implies $i_{\ast}(x) \neq 0$. Therefore $i_{\ast}$ is injective. The proof is complete by induction. 
\end{proof}

\begin{lemma} \label{lem:5homology} Let $n \geq 6$. For $X \in \W_n^6$, $\tilde{H}_5(X) = 0$.
\end{lemma}
\begin{proof}
	Let $X = X_1 \cup \ldots \cup X_k$, where each $X_i$ is the  Vietoris-Rips complex of a $6$-dimensional cube subgraph of $\I_n$.  If $k= 1$, then $X \cong \Delta_6$ and hence result is true by \Cref{thm:adm}. Let $k > 1$ and assume that for any $l < k$ and $i_1, \ldots, i_l \in [k]$, if $X_{i_1} \cup \ldots \cup X_{i_l} \in \W_n^6$, then $\tilde{H}_5(X_{i_1} \cup \ldots \cup X_{i_l}) = 0$. 
	
	If $k = 2$, then from  \Cref{claim:base} $(i)$, there exists  $q_1  \in [k]$ such that $ \bigcup\limits_{j \neq q_1} X_j \in \W_n^6$ and  $  X_{q_1} \cap \bigcup\limits_{j \neq q_1} X_j \in \W_n^{5}$.
	
	Further, if $ k \geq 3$, then from \Cref{claim:base} $(ii)$ 	 there exist $\lambda, q_2  \in [k]$ such that $ \bigcup\limits_{j \neq q_2} X_j \in \W_n^6$, $  X_{q_2} \cap \bigcup\limits_{j \neq q_2} X_j \in \W_n^{5}$ and  for any subset $A \subseteq [k] \setminus \{q_2\}$  containing $\lambda$,
	there exists $p \in A \setminus \{\lambda\}$ such that $   \bigcup\limits_{i \in A \setminus \{p\}} (X_i  \cap X_p \cap X_{q_2}) \neq \emptyset$ implies  $  \bigcup\limits_{i \in A \setminus \{p\}} (X_i  \cap X_p \cap X_{q_2})  \in \W_n^{4}$.

	Without loss of generality we assume that if $k= 2$, then $ q_1  = k$ and if $k \geq  3$, then  $q_2 = k, \lambda = 1$  and for $A = [k-1], p = k-1$. 
	Let $Y = X_1 \cup \ldots \cup X_{k-1}$.   Then by induction hypothesis $\tilde{H}_5(Y) = 0$ and $\tilde{H}_5(X_k) = 0$.
	By Mayer-Vietoris sequence for homology, we have 
	$$ \cdots  \longrightarrow  \tilde{H}_{5}(Y)  \oplus \tilde{H}_{5}(X_p) \longrightarrow \tilde{H}_{5}(X) \longrightarrow  \tilde{H}_{4}(Y\cap X_p) {\overset{ h_{4} }\longrightarrow}  \tilde{H}_{4}(Y)  \oplus \tilde{H}_{4}(X_p) \longrightarrow \cdots  $$
	
	Using   \Cref{lem:injective}, we conclude that   the map $i_{\ast}:\tilde{H}_{4}(Y \cap X_k) \longrightarrow   \tilde{H}_{4}(X_k)$ induced by the inclusion $Y \cap X_k \xhookrightarrow{} X_k$, is injective and therefore the map  $h_4:\tilde{H}_{4}(Y \cap X_k) \longrightarrow   \tilde{H}_{4}(Y)  \oplus \tilde{H}_{4}(X_k)$ is also injective. 	Since 	$\tilde{H}_5(Y) = 0$ and $\tilde{H}_5(X_k) = 0$, we conclude that $\tilde{H}_5(X) = 0$.
\end{proof}

\begin{lemma}\label{thm:main}
	Let $m \geq 7$.  For any $X \in \W_n^m$ and  $j  \in \{5, 6\}$,  $\tilde{H}_j(X) = 0$.
\end{lemma}

\begin{proof}
	Let $X = X_1 \cup  \ldots \cup X_p$, where each $X_i$ is the  Vietoris-Rips complex of an $m$-dimensional cube subgraph of $\I_n$.  	Proof is by induction on $m$ and $p$.
	Let $m = 7$ . We  show that 
	$\tilde{H}_j(X) = 0$ if $j  \in \{5, 6\}$.
	
	Proof is by induction on $p$. If $p = 1$, then $X \simeq \Delta_7$ and therefore  result follows from \Cref{thm:adm}.
	Let $p > 1$. Inductively assume that for any $l < p$ and $i_1, \ldots, i_l \in [p]$, if $X_{i_1} \cup \ldots \cup X_{i_l} \in \W_n^7$, then $\tilde{H}_j(X_{i_1} \cup \ldots \cup X_{i_l}) = 0$ for $j \in \{5, 6\}$. 
	From \Cref{claim:base} $(i)$, there exists $ t \in [p]$ such that $\bigcup\limits_{j \neq t} X_j \in \W_n^{7}$  and $X_t \cap \bigcup\limits_{j\neq t} X_j \in \W_n^{6}$. Without loss of generality assume that $t = p$. Let $Y = X_1 \cup \ldots \cup X_{p-1}$.  Then by induction hypothesis 
	$\tilde{H}_j(Y) = 0$ and $\tilde{H}(X_p) = 0$ for  $j  \in \{5, 6\}$.   By Mayer-Vietoris sequence for homology, we have 
	$$ \cdots  \longrightarrow  \tilde{H}_{j}(Y)  \oplus \tilde{H}_{j}(X_p) \longrightarrow \tilde{H}_{j}(X) \longrightarrow \tilde{H}_{j-1}(Y\cap X_p) {\overset{ h_{j-1}} \longrightarrow}  \tilde{H}_{j-1}(Y)  \oplus \tilde{H}_{j-1}(X_p) \longrightarrow \cdots  $$
	
	Since $Y \cap X_p \in \W_n^6$, 	$\tilde{H}_5(Y \cap X_p) = 0$ by \Cref{lem:5homology}. If $j = 6$, then since  $\tilde{H}_6(Y) = 0, \tilde{H}_6(X_p) = 0 $ and $\tilde{H}_5(Y \cap X_p) = 0 $, we see that $\tilde{H}_6(X) = 0$.   Using \Cref{claim1} $(ii)$ and  \Cref{lem:injective}, we conclude that the map $i_{\ast}:\tilde{H}_{4}(Y \cap X_p) \longrightarrow   \tilde{H}_{4}(X_p)$ induced by the inclusion $Y \cap X_p \xhookrightarrow{} X_p$, is injective and therefore the map  $h_4:\tilde{H}_{4}(Y \cap X_p) \longrightarrow   \tilde{H}_{4}(Y)  \oplus \tilde{H}_{4}(X_p)$ is also injective.  If $j= 5$, then since $\tilde{H}_5(Y) = 0, \tilde{H}_5(X_p) = 0 $, we conclude that   $\tilde{H}_5(X) = 0$. 	Hence result is true for $m = 7$, {\it i.e.,} 	for any $X \in \W_n^7$, 
	$\tilde{H}_j(X) = 0$ for $j  \in \{5, 6\}$. Now let $m \geq 8$.
	
	{\it \large Induction hypothesis 1:} For any $7 \leq l < m$ and $j \in \{5, 6\}$,  if  $X \in \W_n^l$, then 
	$\tilde{H}_j(X) = 0$. 
	
	Let $7 \leq l < m$ and suppose $Z \in \W_n^{l+1}$. Let $Z = Z_1 \cup \ldots \cup Z_q$, where each $Z_i$ is the  Vietoris-Rips complex of an $l+1$-dimensional cube subgraph of $\I_n$. We show that $\tilde{H}_j(Z) = 0$ for $j \in \{5, 6\}$.
	
	Proof is by induction on $q$.	If $q = 1$, then $Z \simeq \Delta_{l+1}$. Since $l \geq 7$,  from \Cref{lem:homology}, any $j$-cycle $c$ in $Z$ is homologous to a $j$-cycle $\tilde{c}$ in $ \partial(Z)$ for $j \in \{5, 6\}$.  Hence it is enough to show that $\tilde{H}_j(\partial(Z)) = 0$ for $j \in \{5, 6\}$.  Observe that
	$\partial(Z) \in \W_n^l$ and therefore by induction hypothesis $1$, 
	$\tilde{H}_j(\partial(Z)) = 0$ for $j \in \{5, 6\}$. 
	Let $q > 1$.
	
	{\it \large Induction hypothesis 2:}  For any $t < q, i_1, \ldots, i_t \in [q]$ and $j \in \{5, 6\}$, if $X_{i_1} \cup \ldots \cup X_{i_t} \in \W_n^{l+1}$, then  $\tilde{H}_j(X_{i_1} \cup \ldots \cup X_{i_t}) = 0$.
	
	From \Cref{claim:base} $(i)$, there exists $ s \in [q]$ such that $\bigcup\limits_{j \neq s} Z_j \in \W_n^{l+1}$  and $Z_s \cap \bigcup\limits_{ j \neq s} Z_j \in \W_n^{l}$. Without loss of generality assume that $s= q$.  	Let $U= Z_{1} \cup \ldots \cup Z_{q-1} $. By induction hypothesis $2$,
	$\tilde{H}_j(U) = 0$ and $\tilde{H}_j(Z_q) = 0$ for $j  \in \{5, 6\}$.  
	By Mayer-Vietoris sequence for homology, we have 
	$$ \cdots  \longrightarrow  \tilde{H}_{j}(U)  \oplus \tilde{H}_{j}(Z_q) \longrightarrow \tilde{H}_{j}(Z) \longrightarrow \tilde{H}_{j-1}(U \cap Z_q) {\overset{ h_{j-1}} \longrightarrow}  \tilde{H}_{j-1}(U)  \oplus \tilde{H}_{j-1}(Z_q) \longrightarrow \cdots  $$
	
	Since  $U\cap Z_q \in \W^l_n$,  from  induction hypothesis  $1$,
	$\tilde{H}_j(U\cap Z_q) = 0$ for  $j\in \{5, 6\}$.	If $j= 6$, then since $\tilde{H}_6(U) = 0, \tilde{H}_6(Z_q) = 0 $ and $\tilde{H}_5(U \cap Z_q) = 0 $, we see that  $\tilde{H}_6(Z) = 0$. If $j = 5$, then since $\tilde{H}_5(U) = 0, \tilde{H}_5(Z_q) = 0 $ and the map $h_4: \tilde{H}_{4}(U \cap Z_q) \longrightarrow   \tilde{H}_{4}(U)  \oplus \tilde{H}_{4}(Z_q)$ is injective by \Cref{lem:injective}, we get that  $\tilde{H}_5(Z) = 0$. 
	
	This completes the proof.
\end{proof}
We are now ready to prove main result of this section. 

\begin{proof}[Proof of \Cref{thm:main3}]
	We must show that for $n \geq 5$, $\tilde{H}_i(\vr{\I_n}{3};\Z) \neq 0$ if and only if $i \in \{4, 7\}$. 	Using  \Cref{thm:collapsibility} and \Cref{prop:collapsibilitysubcomplex}, we see that $\Delta_n$ is homotopy equivalent to a subcomplex of dimension less than $8$. Hence $\tilde{H}_i(\Delta_n) = $ for all $i \geq 8$. Let $X$ be the Vietoris-Rips complex of a $4$-dimensional cube subgraph of $\I_n$. Then using \Cref{lem:retraction}, there exists a retraction  $r:\Delta_n \to X$. Since $\Delta_4 \cong S^7$ and $X \cong \Delta_4$, we see that $\tilde{H}_7(X) \neq 0$.  Further, since $r_{\ast}: \tilde{H}_7(\Delta_n) \to \tilde{H}_7(X)$ is surjective, $\tilde{H}_7(\Delta_n) \neq 0$.  
	
	If $n \leq 6$, then result follows from \Cref{thm:adm}. So assume that $n \geq 7$.
	Since $\Delta_n \in \W_n^n$, \Cref{thm:main} implies that  $\tilde{H}_j(\Delta_n) = 0$ for $j \in\{5, 6\}$.   Let $Y$ be the Vietoris-Rips complex of a $5$-dimensional cube subgraph of $\I_n$. From \Cref{lem:retraction}, there exists a retraction  $r_1:\Delta_n \to Y$. Since $Y \cong \Delta_5$, using \Cref{thm:adm} we conclude that $\tilde{H}_4(\Delta_n) \neq 0$. From \Cref{lem:contractible}, $\tilde{H}_i(\Delta_n) = 0$ for $i \leq 3$. This completes the proof. 
\end{proof}

\section{Future Directions}

In \Cref{thm:main3}, we have shown that $\vr{\I_n}{3}$ has non-trivial homology only in dimensions $4$ and $7$. Further,  the complex $\vr{\I_n}{2}$ is homotopy equivalent to a wedge sum of $3$-spheres.  For $r \in \{2, 3\}$, since $\vr{\I_n}{2}$  has non trivial homology only  in dimension $i \in \{r+1, 2^r-1\}$, we make the following conjecture.
\begin{conj}
	For $n \geq r+2$,  $\tilde{H}_i(\vr{\I_n}{r}; \Z) \neq 0$ if and only if $i \in \{r+1, 2^r-1\}$.
\end{conj} 

The following is a natural question to ask.

\begin{question} Let $n \geq r+2$. Is  $\vr{\I_n}{r}$ homotopy equivalent to a wedge sum of spheres of dimensions $r+1$ and $2^{r}-1$ ?
\end{question}

In Theorems \ref{thm:collapsibility2} and \ref{thm:collapsibility}, we have proved that the collapsibility number of the complex  $\vr{\I_n}{r}$ is $2^r$ for $r \in \{2, 3\}$. The complex $\vr{\I_n}{1}$ is $1$-dimensional and it is  isomorphic to the graph $\I_n$. Hence the collapsibility number of $\vr{\I_n}{1}$ is $2$. Further, it is easy  to check that $\vr{\I_n}{n-1}$ is $(2^{n-1}-1)$-dimensional and it is isomorphic to the join of $2^{n-1}$-copies of $S^0$. Hence the collapsibility number of $\vr{\I_n}{n-1}$ is $2^{n-1}$.  This leads us to make the following conjecture.  

\begin{conj}
	For $n \geq r+1$, the collapsibility number of $\vr{\I_n}{r}$ is $2^r$.
\end{conj}

\section*{Acknowledgements}
I am very grateful to Basudeb Datta
for many helpful discussions. A part of this work was completed when I was  a postdoctoral fellow at IISc Bangalore, India. This work was partially supported by NBHM postdoctoral fellowship of the author. 

\bibliographystyle{abbrv}

\end{document}